\documentclass[11pt,english,bezier,amstex,reqno,amsart]{article}
\usepackage[T1]{fontenc}
\usepackage[latin9]{inputenc}
\usepackage{verbatim}
\usepackage{amsthm}
\usepackage{amsmath}
\usepackage{amssymb}
\usepackage{esint}
\usepackage{srcltx}
\makeatletter
\numberwithin{equation}{section}
\numberwithin{figure}{section}

\usepackage[T1]{fontenc}
\usepackage[latin9]{inputenc}
\usepackage{verbatim}
\usepackage{amsthm}
\usepackage{amssymb}
\usepackage{color}              
\usepackage{epsfig}
\usepackage{graphicx}
\usepackage{amsmath}
\usepackage{amssymb}
\usepackage{epsfig}
\usepackage{hyperref}
\usepackage{amsfonts}
\usepackage{pstricks,pst-node,pst-tree}
\definecolor{MyDarkBlue}{rgb}{0,0.08,0.50}
\definecolor{BrickRed}{rgb}{0.65,0.08,0}

\numberwithin{equation}{section}
\numberwithin{figure}{section}
\theoremstyle{plain}
\newtheorem{thm}{Theorem}
  \theoremstyle{plain}
  \newtheorem{cor}[thm]{Corollary}
  \theoremstyle{remark}
  
  \theoremstyle{plain}
  \newtheorem{lem}[thm]{Lemma}

\usepackage{amsthm}

\numberwithin{equation}{section}
\numberwithin{figure}{section}

\makeatletter
\numberwithin{thm}{section}

\makeatother

\newcommand{\R}{\mathbb{R}}
\newcommand{\E}{\mathbb{E}}
\newcommand{\p}{\mathbb{P}}
\newcommand{\convN}{\underset{N\rightarrow +\infty}{\longrightarrow}}
\newcommand{\1}{\text{\bf 1}}

\theoremstyle{plain}
\newtheorem{assump}{Assumption}

\date\today

\newcommand{\beqnn}{\begin{eqnarray}}
\newcommand{\eeqn}{\end{eqnarray}}
\newcommand{\beq}{\begin{eqnarray*}}
\newcommand{\eeq}{\end{eqnarray*}}
\newcommand{\clb}{\mathcal{B}}
\newcommand{\cls}{\mathcal{S}}
\newcommand{\clt}{\mathcal{T}}

\newcommand{\clp}{\mathcal{P}}
\newcommand{\clf}{\mathcal{F}}
\newcommand{\N}{\mathbb{N}}
\newcommand{\be}{\begin{equation}}
\newcommand{\ee}{\end{equation}}
\newcommand{\eps}{\epsilon}
\newcommand{\ti}{\tilde}

\makeatletter
\newcommand\ackname{Acknowledgements}
\if@titlepage
  \newenvironment{acknowledgements}{%
      \titlepage
      \null\vfil
      \@beginparpenalty\@lowpenalty
      \begin{center}%
        \bfseries \ackname
        \@endparpenalty\@M
      \end{center}}%
     {\par\vfil\null\endtitlepage}
\else
  \newenvironment{acknowledgements}{%
      \if@twocolumn
        \section*{\abstractname}%
      \else
        \small
        \begin{center}%
          {\bfseries \ackname\vspace{-.5em}\vspace{\z@}}%
        \end{center}%
        \quotation
      \fi}
      {\if@twocolumn\else\endquotation\fi}
\fi
\makeatother

\author{A. Budhiraja \footnote{  Department of Statistics and Operations Research, University of North Carolina, Chapel Hill, NC, U.S.A,  {\tt budhiraja@email.unc.edu}}\and
P. Del Moral \footnote{INRIA Bordeaux-Sud Ouest, Bordeaux
Mathematical Institute, Universit\'e Bordeaux I, 351, cours de la
Lib\'eration 33405 Talence cedex, France, {\tt
Pierre.Del\_Moral@inria.fr}}\and S. Rubenthaler
\footnote{Laboratoire de math\'ematiques J.A. Dieudonn\'e,
universit\'e de Nice-Sophia Antipolis, Parc Valrose, 06108 Nice
cedex 02, France,
 {\tt rubentha@unice.fr}}}

\makeatother

\begin{document}

\title{Discrete Time Markovian Agents Interacting Through a Potential}
\maketitle
\begin{abstract}
A discrete time stochastic model for a multiagent system
given in terms of a large collection of interacting Markov chains is studied.
The evolution of the interacting particles is described through a
time inhomogeneous transition probability kernel that depends on the
`gradient' of the potential field. The particles, in turn, dynamically
modify the potential field through their cumulative input. Interacting Markov processes of the above form have been
suggested as models for active biological transport in response to external stimulus such as a chemical gradient.  One of the
basic mathematical challenges is to develop a general theory of stability for such interacting Markovian systems and for the corresponding
nonlinear Markov processes that arise in the large agent limit.  Such a theory would be key to a mathematical understanding of the interactive
structure formation that results from the complex feedback between the  agents and the potential field.  It will also be a crucial ingredient
in developing simulation schemes that are faithful to the underlying model over long periods of time.

The goal
of this work is to study qualitative properties of the above stochastic
system as the number of particles ($N$) and the time parameter ($n$) approach infinity.  In this regard asymptotic
properties of a deterministic nonlinear dynamical system, that arises
in the propagation of chaos limit of the stochastic model, play a
key role. We show that under suitable conditions this dynamical system
has a unique fixed point. This result allows us to study stability
properties of the underlying stochastic model. We show that as $N\to\infty$,
the stochastic system is well approximated by the dynamical system,
{\em uniformly} over time. As a consequence, for an arbitrarily
initialized system, as $N\to\infty$ and  $n\to\infty$,
the potential field and the empirical measure of the interacting particles
are shown to converge to the unique fixed point of the dynamical system.
In general, simulation of such interacting Markovian systems is a
computationally daunting task. We propose a particle based approximation
for the dynamic potential field which allows for a numerically tractable
simulation scheme. It is shown that this simulation scheme well approximates
the true physical system, uniformly over an infinite time horizon.
\end{abstract}

\noindent \textbf{MSC2000 subject classification.} 60J05, 60K35, 92C45,
70K20, 60K40.




\section{Introduction.}

In recent years there has been a significant interest in agent based
modeling for complex systems. Examples of such models abound in physical
and social sciences and include problems of biological aggregation \cite{Ste},
chemotactic response dynamics \cite{FriIgn},  self organized networks \cite{LaNo},
communication systems \cite{GrRo}, opinion dynamics\cite{GGB}, multi-target tracking \cite{caron-del-moral-pace-vo-2010},etc.
See \cite{sch, GrRo} for a comprehensive list of references.
One popular model for interacting multi-agent systems (see \cite{sch}
and references therein) consists of a large number of particles that
influence each other indirectly through a dynamic potential field
and can formally be described through the following system of equations.
\begin{eqnarray}
dX_{i}(t) & = & \nabla h(X_{i}(t),t)dt+dW_{i}(t),\; X_{i}(0)=x_{i}\in\R^{d},i=1,\cdots N.\nonumber \\
\frac{\partial}{\partial t}h(t,x) & = & -\alpha h(t,x)+D\Delta h(t,x)+\beta\sum_{i=1}^{N}g(X_{i}(t),x),\;\; h(0,x)=h_{0}(x).\label{rd1}\end{eqnarray}
 Here $W_{i},i=1\cdots N$ are independent Brownian motions which
drive the state processes $X_{i}$ of the $N$ interacting particles.
The interaction between the particles arises indirectly through the
underlying potential field $h$ which changes continuously according
to a diffusion equation and through the aggregated input of the $N$
particles. One example of such an interaction is in chemotactic cell
response dynamics where cells preferentially move towards  a higher
chemical concentration and themselves release chemicals into the
medium, in response to the local information on the environment,
thus modifying the potential field dynamically over time. In this
context, $h(t,x)$ represents the concentration of a chemical at time
$t$ and location $x$. Diffusion of the chemical in the medium is
captured by the Laplacian in \eqref{rd1} and the constant $\alpha>0$
models the rate of decay or dissipation of the chemical.
Contribution of the agents to the chemical concentration field is
given through the last term in the equation. The function $g$
captures the agent response rules and can be used to model a wide
range of phenomenon \cite{sch}. The first equation in \eqref{rd1}
describes the motion of a particle in terms of a diffusion process
with drift in the direction of the gradient of the chemical
concentration. Many variations of this basic model with applications
to a wide range of fields can be found in \cite{sch}. A precise
mathematical treatment of \eqref{rd1} presents significant technical
obstacles and existing results in literature are limited to
simulation and formal asymptotic approximations of the system. In
the current work we will study a discrete time model which captures
essential features of the dynamics in \eqref{rd1} and is also
amenable to a rigourous mathematical treatment. The time evolution
of the agents will be described through a time inhomogeneous
transition probability kernel, where the kernel at time instant $n$
is determined in terms of the {}``gradient of the potential
field\textquotedbl{} at time instant $n-1$. The agents in turn
affect and contribute to the potential field dynamically over time.
Thus as in the formal continuous time setting described above, the
$N$-agent dynamics is strongly coupled and describes a (time
inhomogeneous) Markov chain in $E^{N}$ where $E$ (a compact subset
of $\R^{d}$) is the state space of a typical agent. Although the
model description is considerably simpler in discrete time, our
objective here is to go beyond formal heuristics (as is the current
state of the art for the continuous time model in \eqref{rd1}) and
to formulate and study precise mathematical properties of the
system. We will establish convergence of the stochastic model to the
solution of a non-linear dynamical system, over an arbitrary fixed
time horizon, as the number of agents approach infinity and as a
consequence obtain a propagation of chaos result(Theorem
\ref{lem:convergence-non-uniforme} and Corollary \ref{propchaos}).
We are particularly interested in the stability of the system as
$N\to\infty$, over long periods of time. A mathematical
understanding of the stability behaviour is key to the study of long
term structure formation resulting from the complex interactions
between the agents and the potential field.  Stability results for
the system are also crucial ingredients for studying the behaviour
of approximate simulation schemes over long intervals of time.
Denoting by $n$ the time parameter, we will give conditions under
which, as $N\to\infty$ and $n\to\infty$ (in any order) the potential
field and the empirical measure of the $N$ particles converges to
limits that are independent of the initial configuration (Corollary
\ref{cor:appfixpt}). These limits are characterized as the unique
fixed point of the limit deterministic non-linear dynamical system
(Theorem \ref{cor:unique-fixed-point}). Uniform in time convergence
of the stochastic model to the non-linear deterministic dynamics is
established as well (Theorem \ref{pro:unif-conv-theoretical}). In
general, simulation of interacting Markovian systems is a
computationally daunting task. We propose a particle based
approximation for the dynamic potential field which allows for a
numerically tractable simulation scheme. Using the above stability
results we show that this simulation scheme well approximates the
true physical system, uniformly over an infinite time horizon
(Theorem \ref{thm:dist-uniform}).

Before we give a formal description of the model, we list some common
notation that will be used in this work. For a Polish space $\cls$,
$\clb(\cls)$ will denote the Borel sigma field on $\cls$ and $\clp(\cls)$
the space of probability measures on $(\cls,\clb(\cls))$. For $x\in\cls$,
$\delta_{x}$ will denote the element in $\clp(\cls)$ that puts unit
mass at the point $x$. For $\mu\in\clp(\cls)$ and a $\mu$-integrable
real measurable map $f$ on $(\cls,\clb(\cls))$, we denote $\int_{\cls}fd\mu$
as $\mu(f)$ or $\langle\mu,f\rangle$. Similar notations will be
used for signed measures. For $\mu\in\clp(\cls)$, $S_{N}(\mu)$ denotes
a random measure defined as $\frac{1}{N}\sum_{k=1}^{N}\delta_{\xi_{k}}$,
where $\{\xi_{k}\}_{k=1}^{N}$ are i.i.d $\cls$ valued random variables.
We denote the space of real bounded measurable maps on $(\cls,\clb(\cls))$
as $\mbox{BM}(\cls)$ and for $f\in\mbox{BM}(\cls)$, define $\Vert f\Vert_{\infty}=\sup_{x\in\cls}|f(x)|$.
The space $\{f\in\mbox{BM}(\cls):\Vert f\Vert_{\infty}\le1\}$ will
be denoted as $\clb_{1}(\cls)$. For a signed measure $\mu$ on $(S,\clb(S))$,
we define the total variation norm of $\mu$ as $\sup_{f\in\clb_{1}(\cls)}|\mu(f)|$
and denote it by $\Vert\mu\Vert_{TV}$. A real function $f$ on $\cls$
is said to be $b$-Lipschitz if, for all $x,y\in\cls$, $|f(x)-f(y)|\le b\; d(x,y)$,
where $d$ is the metric on $\cls$. A transition probability kernel
(also referred to as a Markov kernel) on $S$ is a map $\clt$ from
$S\times\clb(S)$ to $[0,1]$ such that for all $x\in S$, $\clt(x,\cdot)\in\clp(\cls)$
and for all $A\in\clb(\cls)$, $\clt(\cdot,A)\in\mbox{BM}(\cls)$.
For $f\in\mbox{BM}(\cls)$ and a transition probability kernel $\clt$
on $\cls$, define $\clt f\in\mbox{BM}(\cls)$ as $\clt f(\cdot)=\int_{\cls}f(y)\clt(\cdot,dy)$.
For a closed subset $\cls_{0}$ of $\cls$, $\mu\in\clp(\cls_{0})$
and a transition kernel $\clt$ on $\cls$, we define $\mu\clt\in\clp(\cls)$
as $\mu\clt(A)=\int_{\cls_{0}}\clt(x,A)\mu(dx)$. For $d\ge1$, $\R^{d}$
will denote the $d$-dimensional Euclidean space, the standard norm
on which will be denoted by $|\cdot|$. $\N$ {[}resp. $\N_{0}${]}
will denote the space of positive {[}resp. nonnegative{]} integers.
Cardinality of a finite set $G$ will be denoted by $|G|$.

\subsection{Model for Interacting Markovian Agents.}

The system consists of $N$ particles whose states at time $k\in\N_{0}$
denoted as $X_{1}(k),\cdots X_{N}(k)$ take value in a compact set
$E\subset\R^{d}$ with a non-empty interior. Given that the $j$-th
particle is at location $x$ at time instant $k-1$, it transitions,
independently of other particles, to a set $A\subset\clb(E)$ with
probability $M^{\eta_{k-1}}(x,A)$. Here $M^{\eta_{k-1}}$ is a transition
probability kernel determined by a nonnegative function $\eta_{k-1}$
which represents the {}``potential field'' at time instant $k-1$.
Specifically, for any nonnegative function $\Psi:E\rightarrow\R^{+}$,
the Markov kernel $M^{\Psi}$ is defined as follows. 
\be M^{\Psi}(x,dy)=Q(x,dy)e^{-\lambda(\Psi(x)-\Psi(y))_{+}}\\
 +Q_{0}(x,dy)\left(1-\ensuremath{\int}_{z\in E}Q(x,dz)e^{-\lambda(\Psi(x)-\Psi(z))_{+}}\right).
\ee 
 Here $\lambda\in(0,\infty)$ and $Q$, $Q_{0}$ are two transition
probability kernels on $E$. Roughly speaking, to generate a sample
from $M^{\Psi}(x,\cdot)$ one follows the following steps:
\begin{itemize}
\item A sample $Y$ is drawn from $Q(x,\cdot)$ and a sample $\tilde{Y}$
is drawn from $Q_{0}(x,\cdot)$.
\item If $\Psi(Y)\ge\Psi(x)$ the sample point $Y$ is accepted. This corresponds
to preferential motion in the direction of the gradient of the potential
field.
\item If $\Psi(Y)<\Psi(x)$ we accept the sample point $Y$ with probability
$e^{-\lambda(\Psi(x)-\Psi(Y))}$ and take $\tilde{Y}$ with probability
$1-e^{-\lambda(\Psi(x)-\Psi(Y))}$.
\end{itemize}
The kernel $Q$ captures particle dynamics in absence of the potential
field, while $Q_{0}$ can be regarded as the perturbation to the nominal
dynamics against the concentration gradient, caused by the potential
field. Note that the effect of $Q_{0}$ increases as $\lambda$ becomes
larger.

We now describe the evolution of the potential field and its interaction
with the particle system. This evolution will aim to capture the essential
features of the PDE in \eqref{rd1} which are: diffusion, dissipation
and dynamic agent input. Let $P$ and $P'$ be transition
probability kernels on $\R^{d}$ having density with respect to some
fixed reference measure $\ell$ on $\R^{d}$ (for example, the Lebesgue
measure). Throughout, we will speak of densities as with respect to
$\ell$. 
Given $m\in\clp(E)$, define a transition probability kernel $R_{m}$
on $\R^{d}$ as \begin{equation}
R_{m}(x,A)=(1-\eps)P(x,A)+\eps mP'(A),\;\; x\in\R^{d},A\in\clb(\R^{d}).\label{transR}\end{equation}
 Then given that the potential field at time $k-1$ is described by
a nonnegative function $\eta_{k-1}$ on $\R^{d}$ satisfying $\int_{\R^{d}}\eta_{k-1}(x)dx=1$,
and the state values of the $N$ particles are $X_{i}(k-1)$, $i=1,\cdots N$,
$\eta_{k}$ is defined by the relation. \begin{equation}
\eta_{k}(y)=\int_{\R^{d}}\eta_{k-1}(x)R_{m_{k-1}}(x,y)\ell(dx),\; m_{k-1}=\frac{1}{N}\sum_{j=1}^{N}\delta_{X_{j}(k-1)}.\label{PFiel}\end{equation}
 In this description, diffusion and advection of the chemical is captured by the
kernel $P$, dissipation by the factor $(1-\eps)$ and contributions
by the agents through the term $\eps m_{k-1}P'$. We remark that unlike
the continuous time setting, here we introduce a single parameter
$\eps$ rather than two distinct parameters $\alpha$ and $\beta$.
This parametrization ensures that the concentration function $\eta_{k}$,
for each $k$ is a probability density on $\R^{d}$. The more general
setting can be considered as well (although not pursued here) by considering
suitably normalized concentration fields.

Thus summarizing, analogous to \eqref{rd1}, the coupled system of
equations describing the evolution of the potential field and particle
states is given as follows. Denote by $\clp^{*}(\R^{d})$ the space
of probability measures on $\R^{d}$ that are absolutely continuous
with respect to $\ell$. We will identify an element in $\clp^{*}(\R^{d})$
with its density and denote both by the same symbol. For $N\ge1$,
define $\Pi^{N}:E^{N}\to\clp(E)$ by the relation \begin{equation}
\Pi^{N}(y_{1},\cdots y_{N})=\frac{1}{N}\sum_{i=1}^{N}\delta_{y_{i}},\;\;(y_{1},\cdots y_{N})\in E^{N}.\label{emp1}\end{equation}
 Fix random variables $X_{1},\cdots,X_{N}$ with values in $E$ and $\eta_{0}^{N}\in\clp^{*}(\R^{d})$.
Let $m_{0}^{N}=\Pi^{N}(X_{1},\cdots,X_{N})$. The interacting system
of particles and the potential field is described as a family $(X^{N}(k),m_{k}^{N},\eta_{k}^{N})_{k\in\N_{0}}$
of $E^{N}\times\clp(E)\times\clp^{*}(\R^{d})$ valued random variables
on a probability space $(\Omega,\clf,\p)$, defined recursively as
follows. Let $X^{N}(0)=(X_{1},\cdots,X_{N})$ and define $\clf_{0}=\sigma\{X^{N}(0)\}$.
For $k\ge1$

\begin{equation}
\begin{cases}
\p(X^{N}(k)\in A\mid\clf_{k-1})=\bigotimes_{j=1}^{N}\left(\delta_{X_{j}^{N}(k-1)}M^{\eta_{k-1}^{N}}\right)(A),\; A\in\clb(E^{N})\\
m_{k}^{N}=\Pi^N(X^N(k)) = \frac{1}{N}\sum_{j=1}^{N}\delta_{X_{j}^{N}(k)}\\
\eta_{k}^{N}=\eta_{k-1}^{N}R_{m_{k-1}^{N}}\\
\mathcal{F}_{k}=\sigma(\eta_{k}^{N},X^{N}(k))\vee\clf_{k-1}\,.\end{cases}\label{eq:def-theoretical-model}\end{equation}
 Along with the $N$ particle system we will also consider the non-linear
Markov model which formally corresponds to the $N\to\infty$ limit
of \eqref{eq:def-theoretical-model}. Define the map $\Phi:\clp(E)\times\clp^{*}(\R^{d})\to\clp(E)\times\clp^{*}(\R^{d})$
by the relation \[
\Phi(m,\eta)=(mM^{\eta},\eta R_{m}),\;(m,\eta)\in\clp(E)\times\clp^{*}(\R^{d}).\]
 Define a sequence of probability measures $(m_{n},\eta_{n})$ on
$\R^{d}$ by the following recurrence formula. Fix $(m_{0},\eta_{0})\in\clp(E)\times\clp^{*}(\R^{d})$.
For $n\in\N_{0}$ \begin{align}
m_{n+1} & =m_{n}M^{\eta_{n}}\,,\label{eq:rec-m}\\
\eta_{n+1} & =\eta_{n}R_{m_{n}}\,.\label{eq:rec-eta}\end{align}
 In other words, \begin{equation}
(m_{n+1},\eta_{n+1})=\Phi(m_{n},\eta_{n})\,.\label{eq:rec01}\end{equation}
Such a coupled system is similar to what can be found in \cite{caron-del-moral-doucet-pace-2010,caron-del-moral-pace-vo-2010}.

\subsection{Main Results.}

We now summarize the main results of this work. For measures $(m,\eta),(m',\eta')\in\clp(E)\times\clp^{*}(\R^{d})$,
we define the norm \[
\Vert(m,\eta)-(m',\eta')\Vert=\Vert m-m'\Vert_{TV}+\Vert\eta-\eta'\Vert_{TV}\,.\]
 We begin with an assumption on the kernels $P,P'$. \begin{assump}
\label{assu:forthm1} The kernel $P'$ is uniformly bounded on $\R^{d}\times\R^{d}$
and $P'(\cdot,y)$ is Lipschitz, uniformly in $y\in\R^{d}$, namely,
for some $M_{P'},l_{P'}\in(0,\infty)$ \[
\sup_{(x,y)\in\R^{d}\times\R^{d}}P'(x,y)\le M_{P'},\;\sup_{y\in\R^{d}}|P'(x,y)-P'(x',y)|\le l_{P'}|x-x'|,\;\forall x,x'\in\R^{d}.\]
 The kernel $P$ satisfies, for some $\bar{M}_{P}\in(0,\infty)$, \[
\sup_{y\in\R^{d}}\int_{x\in\R^{d}}P(x,y)\ell(dx)\leq\bar{M}_{P}\,.\]
 \end{assump} The following result establishes the convergence of
the stochastic system to the non-linear deterministic dynamical system
over any fixed time horizon.
\begin{thm} \label{lem:convergence-non-uniforme}
Suppose that Assumption \ref{assu:forthm1} holds. Also suppose that
\[||\eta_{0}^{N}-\eta_{0}||_{\infty}\to0 \mbox{ and } \sup_{f\in\mathcal{B}_{1}(E)}\mathbb{E}(|\langle m_{0}-m_{0}^{N},f\rangle|)\to 0\;
{ as }\;\ N\to\infty.\] Then, $\forall k\in\N_{0}$, \begin{equation}
\sup_{f\in\clb_{1}(E)}\E(|\langle m_{k}^{N}-m_{k},f\rangle|+\Vert\eta_{k}^{N}-\eta_{k}\Vert_{\infty})\underset{N\rightarrow+\infty}{\longrightarrow}0\,.\label{toshow}\end{equation}
 \end{thm} We remark that $\E\Vert\eta_{k}^{N}-\eta_{k}\Vert_{\infty}\to0$
implies, by an application of Scheffe's theorem, that $\E\Vert\eta_{k}^{N}-\eta_{k}\Vert_{TV}\to0$.
This observation will be used in the proof of Theorem \ref{pro:unif-conv-theoretical}

As an immediate consequence of this result we obtain the following
{\em propagation of chaos} result. For $p\ge1$, denote by $m_{k}^{N,p}$
the probability law of $(X_{1}^{N}(k),\cdots X_{p}^{N}(k))$ on $E^{p}$.
\begin{cor} \label{propchaos} Under Assumption \ref{assu:forthm1},
for every $p\in\N$, $m_{k}^{N,p}$ converges weakly to $m_{k}^{\otimes p}$,
as $N\to\infty$. \end{cor} We now study time asymptotic properties
of the system. We begin with two basic assumptions. The first is on
the kernel $Q$.
\begin{assump}\label{assump:Q}There exist $\epsilon_{Q}\in(0,1)$
and $\ell_{1}\in\clp(E)$ such that $\forall x\in E$, $\forall$
$A\in\clb(E)$, $Q(x,A)\geq\epsilon_{Q}\ell_{1}(A)$.\end{assump}
It is well known that under Assumption \ref{assump:Q}, $Q$ is $(1-\epsilon_{Q})$-contracting
for the total variation norm (see Lemma \ref{lem:dobrushin} in Appendix
for a proof), i.e \[
\Vert\mu Q-\mu'Q\Vert_{TV}\le(1-\eps_{Q})\Vert\mu-\mu'\Vert_{TV},\forall\mu,\mu'\in\clp(E).\]
 Next, we will make the following assumption on the kernels $P,P'$.
\begin{assump}\label{assump:P}
\begin{enumerate}
\item \label{assump:P-01} There exists $\beta(P')\in(0,1)$ such that for
all $m,m'\in\clp(E)$, \[
\Vert mP'-m'P'\Vert_{TV}\leq\beta(P')\Vert m-m'\Vert_{TV}.\]

\item The kernel $P$ is uniformly bounded on $\R^{d}\times\R^{d}$, i.e.
for some $M_{P}\in(0,\infty)$ \[
\sup_{(x,y)\in\R^{d}\times\R^{d}}P(x,y)\le M_{P}.\]

\end{enumerate}
\end{assump} We begin with the following result on the fixed points
of the dynamical system \eqref{eq:rec01}.
\begin{thm} \label{cor:unique-fixed-point}
Suppose that Assumptions \ref{assu:forthm1}, \ref{assump:Q} and
\ref{assump:P} hold. Then there are $\lambda_{0},\eps_{0}\in(0,\infty)$
such that for all $\lambda\le\lambda_{0}$ and $\epsilon\le\eps_{0}$, $\Phi$
has a unique fixed point in $\mathcal{P}(E)\times\mathcal{P}^{*}(\R^{d})$.
\end{thm} Recall the sequence $(m_{n},\eta_{n})$ defined through
equation \eqref{eq:rec01} recursively with a fixed initial pair $(m_{0},\eta_{0})\in\clp(E)\times\clp^{*}(\R^{d})$.
Also recall the collection of random variables $(X^{N}(k),m_{k}^{N},\eta_{k}^{N})_{k\in\N_{0}}$
with values in $E^{N}\times\clp(E)\times\clp^{*}(\R^{d})$ defined
in \eqref{eq:def-theoretical-model} starting with a fixed $\eta_{0}^{N}\in\clp^{*}(\R^{d})$
and $X_{1},\dots,X_{N}\in E$.

We now consider uniform in time convergence of the stochastic system
to the non-linear dynamical system. For that we will make the following
additional assumptions. \begin{assump}\label{leftover}
\begin{enumerate}
\item $P(\cdot,y)$ is Lipschitz uniformly in $y\in\R^{d}$ namely, for
some $l_{P}\in(0,\infty)$ \[
\sup_{y\in\R^{d}}|P(x,y)-P(x',y)|\le l_{P}|x-x'|,\;\forall x,x'\in\R^{d}.\]

\item \label{enu:lip-P-01}There exist $p\in\mathcal{P}(\R^{d})$ and $\ti l_{P'}$
such that $\forall x,x'\in E$, $A\in\clb(\R^{d})$, \[
|P'(x,A)-P'(x',A)|\leq\ti l_{P'}|x-x'|p(A).\]

\item $P(x,\cdot)$, $P'(x,\cdot)$ are Lipschitz uniformly in $x\in\R^{d}$
namely, for some $\bar{l}_{P.P'}\in(0,\infty)$ \[
\sup_{x\in\R^{d}}\max\{|P(x,y)-P(x,y')|,|P'(x,y)-P'(x,y')|\}\le\bar{l}_{P,P'}|y-y'|,\;\forall y,y'\in\R^{d}.\]

\item For some $\bar{M}_{P,P'}\in(0,\infty)$ \[
\sup_{y\in\R^{d}}\max\{\int_{x\in\R^{d}}P(x,y)l(dx),\int_{x\in\R^{d}}P'(x,y)l(dx)\}\leq\bar{M}_{P,P'\,.}\]

\end{enumerate}
\end{assump}

Denote by $\ell_{E}$ the restriction of the measure $l$ to $E$.

\begin{assump}\label{assump:lipQ} For all $x\in E$, $Q(x,.)$ and
$Q_{0}(x,.)$ have densities with respect to $\ell_{E}$. The densities
are bounded on $E\times E$, namely for some $M_{Q,Q_{0}}\in(0,\infty)$
\[
\sup_{(x,y)\in E\times E}\max\{Q(x,y),Q_{0}(x,y)\}\le M_{Q,Q_{0}}.\]
 Furthermore, for some $l_{Q,Q_{0}}\in(0,\infty)$ and $\bar{p}\in\clp(E)$
\begin{multline*}
\max\{|Q(x,A)-Q(x',A)|,|Q_{0}(x,A)-Q_{0}(x',A)|\}\le l_{Q,Q_{0}}|x-x'|\bar{p}(A)\;\forall x,x'\in E,A\in\clb(E).\end{multline*}
 \end{assump}
\begin{thm} \label{pro:unif-conv-theoretical}Suppose
that Assumptions \ref{assu:forthm1} \ref{assump:Q}, \ref{assump:P},
\ref{leftover} and \ref{assump:lipQ} hold. Let $\eps_{0},\lambda_{0}$
be as in Theorem \ref{cor:unique-fixed-point}. Then, whenever $\epsilon\le\eps_{0}$ and
$\lambda\le\lambda_{0}$, we have:\\

(i) For some $c_{0}\in(0,\infty)$, \begin{multline*}
\forall\delta>0\,,\,\exists N_{0},n_{0}\in\N\mbox{ such that }\forall n\ge n_{0}\mbox{ and }N\ge N_{0}\,,\,\sup_{\Vert f\Vert_{\infty}\leq1}\E(|\langle m_{n}^{N}-m_{n},f\rangle|+\Vert\eta_{n}^{N}-\eta_{n}\Vert_{TV})\leq c_{0}\delta\,.\end{multline*}
 (ii) If $||\eta_{0}^{N}-\eta_{0}||_{\infty}\to0$ and $\sup_{f\in\mathcal{B}_{1}(E)}\mathbb{E}(|\langle m_{0}-m_{0}^{N},f\rangle|)\to0$
as $N\to\infty$, then for some $c_{1}\in(0,\infty)$ \begin{multline*}
\forall\delta>0\,,\,\exists N_{0}\in\N\mbox{ such that }\forall n\mbox{ and }N\ge N_{0}\,,\,\sup_{\Vert f\Vert_{\infty}\leq1}\E(|\langle m_{n}^{N}-m_{n},f\rangle|+\Vert\eta_{n}^{N}-\eta_{n}\Vert_{TV})\leq c_{1}\delta\,.\end{multline*}
 \end{thm} As an immediate consequence of Theorems \ref{cor:unique-fixed-point}
and \ref{pro:unif-conv-theoretical} we have that under suitable conditions
$(m_{k}^{N},\eta_{k}^{N})$ approaches the unique fixed point of $\Phi$
as $k\to\infty$ and $N\to\infty$. Namely,
\begin{cor} \label{cor:appfixpt}
Suppose that Assumptions \ref{assu:forthm1} \ref{assump:Q}, \ref{assump:P},
\ref{leftover} and \ref{assump:lipQ} hold. Let $\eps_{0},\lambda_{0}$
be as in Theorem \ref{cor:unique-fixed-point}. Fix $\eps\in(0,\eps_{0})$
and $\lambda\in(0,\lambda_{0})$. Denote the corresponding unique
fixed point of $\Phi$ by $(m_{\infty},\eta_{\infty})$. Then, 
 $ $\begin{eqnarray*}
 &  & \limsup_{n\to\infty}\limsup_{N\to\infty}\sup_{\Vert f\Vert_{\infty}\leq1}\E(|\langle m_{n}^{N}-m_{\infty},f\rangle|+\Vert\eta_{n}^{N}-\eta_{\infty}\Vert_{TV})\\
 & = & \limsup_{N\to\infty}\limsup_{n\to\infty}\sup_{\Vert f\Vert_{\infty}\leq1}\E(|\langle m_{n}^{N}-m_{\infty},f\rangle|+\Vert\eta_{n}^{N}-\eta_{\infty}\Vert_{TV})\\
 & = & 0.\end{eqnarray*}
 \end{cor} Simulation of the stochastic system in \eqref{eq:def-theoretical-model}
or numerical computation of the paths of the deterministic dynamical
system \eqref{eq:rec01} can in general be quite hard and so it is
of interest to develop good approximation schemes. One flexible and
appealing approach is to approximate the measures, $\eta_{n}^{N}$
in the first case and the measures $(\eta_{n},m_{n})$ in the second
case, by discrete probability distributions through a collection of
evolving particles. We will consider one such particle scheme in this
work and show that, under conditions, the error between the dynamical
system in \eqref{eq:rec01} and the one obtained through a particle
approximation can be controlled uniformly in time. We will also prove
a similar result for the error between the actual physical stochastic
system and the one obtained through particle approximations of $\eta_{n}^{N}$.
For these results we will make the following additional assumption.
\begin{assump}\label{assump:gaussian} $P,P'$ are Gaussian kernels
and $\ell$ is the Lebesgue measure on $\R^{d}$. \end{assump} Note
that if $P,P'$ satisfy Assumption \ref{assump:gaussian}, then they
also satisfy Assumptions \ref{assu:forthm1}, \ref{assump:P} and
\ref{leftover} (for Assumption \ref{assump:P} see Lemma \ref{lem:consistancy-of-assumptions}).

We propose the following particle system for the system \eqref{eq:rec01}
that starts from $(m_{0},\eta_{0})\in\clp(E)\times\clp^{*}(\R^{d})$.
We assume that we can draw samples from $(m_{0},\eta_{0})$. We remark
that although for Gaussian kernels, the integral in \eqref{PFiel}
can be computed analytically as a mixture of Gaussian densities,
the computation becomes numerically unfeasible since the number of
terms in this mixture grows linearly over time.

Denote by $(\tilde{X}_{1}^{N}(0),\cdots\tilde{X}_{N}^{N}(0))$ a sample
of size $N$ from $m_{0}$. Let $\tilde{m}_{0}^{N}=\Pi^{N}(\tilde{X}_{1}^{N}(0),\cdots\tilde{X}_{N}^{N}(0))$.
The approximating particle system is given as a family $(\ti X^{N}(k),m_{k}^{N},\eta_{k}^{N})_{k\in\N_{0}}$
of $E^{N}\times\clp(E)\times\clp^{*}(\R^{d})$ valued random variables
on some probability space $(\Omega,\clf,\p)$, defined recursively
as follows. Set $\ti X^{N}(0)=(\tilde{X}_{1}^{N}(0),\cdots\tilde{X}_{N}^{N}(0))$,
$\ti\eta_{0}^{N}=\eta_{0}$, $\clf_{0}=\sigma(\tilde{X}^{N}(0))$.
For $k\ge1$ \begin{equation}
\begin{cases}
\p(\ti X^{N}(k)\in A\mid\clf_{k-1})=\bigotimes_{j=1}^{N}\left(\delta_{\ti X_{j}^{N}(k-1)}M^{\ti\eta_{k-1}^{N}}\right)(A),\; A\in\clb(E^{N})\\
\ti m_{k}^{N}=\frac{1}{N}\sum_{j=1}^{N}\delta_{\ti X_{j}^{N}(k)}\\
\ti\eta_{k}^{N}=(1-\epsilon)(S^{N}(\ti\eta_{k-1}^{N})P)+\epsilon(\ti m_{k-1}^{N}P')\\
\mathcal{F}_{k}=\sigma(\tilde{X}^{N}(k),\tilde{\eta}_{k}^{N})\vee\clf_{k-1}\,.\end{cases}\label{eq:def-scheme}\end{equation}
 Here $S^{N}(\tilde{\eta}_{k-1}^{N})$ is the random probability measure
defined as $\frac{1}{N}\sum_{i=1}^{N}\delta_{Y_{i}^{N}(k)}$ where
$Y_{1}^{N}(k),\cdots Y_{N}^{N}(k)$, conditionally on $\clf_{k-1}$,
are i.i.d. distributed according to $\tilde{\eta}_{k-1}^{N}$.

Notice that under the above Gaussian assumption, $\ti\eta_{k}^{N}$
is a mixture of $2N$- Gaussian random variables for any $k\geq1$,
so we can compute numerically its density at any point. So (\ref{eq:def-scheme})
defines an implementable particle scheme. 
We first consider convergence over a fixed time horizon.

\begin{thm} \label{particlefin} Under Assumption \ref{assump:gaussian}, we have for all $k$, \begin{equation}
\sup_{f\in\clb_{1}(E)}\E(|\langle\ti m_{k}^{N}-m_{k},f\rangle|+\Vert\ti\eta_{k}^{N}-\eta_{k}\Vert_{TV}+\Vert\ti\eta_{k}^{N}-\eta_{k}\Vert_{\infty})\convN0\,.\label{eq:conv-particle-scheme}\end{equation}
 \end{thm} We now consider uniform in time convergence of the approximation
scheme. \begin{thm} \label{particleinf} Suppose that Assumptions
\ref{assump:Q}, \ref{assump:lipQ} and \ref{assump:gaussian} hold.
Then there exists $c_{1}\in(0,\infty)$ such that for $\eps\le\eps_{0}$,
$\lambda\le\lambda_{0}$ and $(\eta_{0},m_{0})\in\clp^{*}(\R^{d})\times\clp(E)$,
the sequences defined in (\ref{eq:def-scheme}) are such that \[
\forall\delta>0\,,\,\exists N_{0}\mbox{ such that }\forall n\mbox{ and }N\ge N_{0}\,,\,\sup_{f\in\clb_{1}(E)}\E(|\langle\ti m_{n}^{N}-m_{n},f\rangle|+\Vert\ti\eta_{n}^{N}-\eta_{n}\Vert_{TV})<c_{1}\delta\,.\]

\label{physsyst} \end{thm} %
{}As corollaries of Theorems \ref{lem:convergence-non-uniforme}, \ref{pro:unif-conv-theoretical},
\ref{particlefin} and \ref{particleinf} we have the following results.
\begin{thm} \label{thm:dist-non-uniform} Under Assumption \ref{assump:Q} and assumptions of Theorem
\ref{lem:convergence-non-uniforme}, we have for all $k\geq0,$ \[
\sup_{f\in\mathcal{B}_{1}(E)}\E(|\langle m_{k}^{N}-\tilde{m}_{k}^{N},f\rangle|+\Vert\eta_{k}^{N}-\tilde{\eta}_{k}^{N}\Vert_{TV})\convN0\,.\]

\end{thm}

\begin{thm} \label{thm:dist-uniform} Under assumptions \ref{assump:Q},
\ref{assump:lipQ} and \ref{assump:gaussian} 
there exists $c_{1}\in(0,\infty)$ such that for $\eps\le\eps_{0}$,
$\lambda\le\lambda_{0}$, \[
\forall\delta>0\,,\,\exists N_{0},n_{0}\in\N\mbox{ such that }\forall n\ge n_{0}\mbox{ and }N\ge N_{0}\,,\,\sup_{\Vert f\Vert_{\infty}\leq1}\E(|\langle m_{n}-\tilde{m}_{n}^{N},f\rangle|+\Vert\eta_{n}^{N}-\tilde{\eta}_{n}^{N}\Vert_{TV})<c_{1}\delta\,.\]
\end{thm}

We remark that in approximating $\eta_{k}$ by $\tilde{\eta}_{k}^{N}$
and $m_{k}$ by $\tilde{m}_{k}^{N}$ we have taken for simplicity
$N$ particles for both approximations. Although not pursued here,
one can similarly analyse particle schemes where the number of particles
for approximating $\eta_{k}$ is different from that used for approximating
$m_{k}$.

\section{Proofs.}


\subsection{Convergence over a finite time horizon.\label{sec:Convergence-of-the-theoretical}}

In this subsection we prove Theorem \ref{lem:convergence-non-uniforme}
and Corollary \ref{propchaos}. We begin with a lemma that will be
used several times in this work. Proof is immediate from Ascoli-Arzela
theorem. \begin{lem} \label{rem:ascoli} Let $K$ be a compact subset
of $\R^{d}$ and let for $a,b\in(0,\infty)$, $F_{a,b}(K)$ be the
collection of all functions $f:K\rightarrow\R$ such that $\Vert f\Vert_{\infty}\leq a$
and $f$ is $b$-Lipschitz. Then for every $\delta>0$ there exists
a finite subset $F_{a,b}^{\delta}(K)$ of $F_{a,b}(K)$ such that
for every signed measure $\mu$ on $K$ \[
\sup_{f\in F_{a,b}(K)}|\langle\mu,f\rangle|)\leq\max_{g\in F_{a,b}^{\delta}(K)}|\langle\mu,g\rangle|+\delta\Vert\mu\Vert_{TV}\,.\]
 \end{lem} Frequently, when clear from context, we will suppress
$K$ in the notation when writing $F,F^{\delta}$. The following elementary
estimate will be used several times. \begin{lem} \label{lem:expo-lipschitz}For
all $x,y,x',y'\in\R$, $ $ \[
\left|e^{-\lambda(x-y)_{+}}-e^{-\lambda(x'-y')_{+}}\right|\leq\lambda|x-x'|+\lambda|y-y'|\,.\]
 \end{lem}

\begin{proof}[Proof of Theorem \ref{lem:convergence-non-uniforme}]
We proceed recursively. Note that, by assumption, \eqref{toshow}
holds for $k=0$. Suppose now that \eqref{toshow} holds for some
fixed $k\in\N_{0}$. Then \[
m_{k+1}^{N}-m_{k+1}=m_{k+1}^{N}-m_{k}^{N}M^{\eta_{k}^{N}}+m_{k}^{N}M^{\eta_{k}^{N}}-m_{k}^{N}M^{\eta_{k}}+m_{k}^{N}M^{\eta_{k}}-m_{k}M^{\eta_{k}}\,.\]
 From Lemma \ref{lem:approx-empirique} in the Appendix \[
\sup_{f\in\mathcal{B}_{1}(E)}\E(|\langle m_{k+1}^{N}-m_{k}^{N}M^{\eta_{k}^{N}},f\rangle|)\leq\frac{2}{\sqrt{N}}\convN0\,.\]
 Note that for all $f\in\clb_{1}(E)$, $m\in\clp(E)$ and $\eta,\eta'\in\clp^{*}(\R^{d})$
\begin{multline}
|mM^{\eta}(f)-mM^{\eta'}(f)|\\
=\left|\int_{x,y\in E}m(dx)Q(x,dy)(e^{-\lambda(\eta(x)-\eta(y))_{+}}-e^{-\lambda(\eta'(x)-\eta'(y))_{+}})f(y)\right.\\
+\int m(dx)Q_{0}(x,dy)\\
\times\left.\left(\int_{z\in E}Q(x,dz)(-e^{-\lambda(\eta(x)-\eta(z))_{+}}+e^{\lambda(\eta'(x)-\eta'(z))_{+}}\right)f(y)\right|\,.\label{eq:rec-maj-04}\end{multline}
 Thus by Lemma \ref{lem:expo-lipschitz} and since \eqref{toshow}
holds for $k$, we have, $\forall f\in\mathcal{B}_{1}(E)$, \[
\E\left(|\langle m_{k}^{N}M^{\eta_{k}^{N}}-m_{k}^{N}M^{\eta_{k}},f\rangle|\right)\leq4\lambda\E\left(\Vert\eta_{k}^{N}-\eta_{k}\Vert_{\infty}\right)\convN0\,.\]
 Also, for $f\in\mathcal{B}_{1}(E)$,\begin{eqnarray*}
\langle m_{k}^{N}M^{\eta_{k}}-m_{k}M^{\eta_{k}},f\rangle & = & \langle m_{k}^{N}-m_{k},M^{\eta_{k}}f\rangle\,.\end{eqnarray*}
 %
{}{} Since $f\in\clb_{1}(E)$, we have that $M^{\eta_{k}}f\in\clb_{1}(E)$
and so %
{}{} \begin{eqnarray*}
\E(|\langle m_{k}M^{\eta_{k}}-m_{k}^{N}M^{\eta_{k}},f\rangle|) & \leq & \sup_{g\in\mathcal{B}_{1}(E)}\E(|\langle m_{k}-m_{k}^{N},g\rangle)\,.\end{eqnarray*}
 Once again using the fact that \eqref{toshow} holds for $k$, we
have from the above inequality, that \[
\sup_{f\in\mathcal{B}_{1}(E)}\E(|\langle m_{k}M^{\eta_{k}}-m_{k}^{N}M^{\eta_{k}},f\rangle|)\convN0\,,\]
 and combining the above convergence statements \[
\sup_{f\in\mathcal{B}_{1}(E)}\E(|\langle m_{k+1}^{N}-m_{k+1},f>|)\convN0\,.\]
 Next \[
\eta_{k+1}^{N}(x)-\eta_{k+1}(x)=(1-\epsilon)(P\star(\eta_{k}^{N}-\eta_{k}))(x)+\epsilon(P'\star(m_{k}^{N}-m_{k}))(x),\; x\in\R^{d},\]
 where for $\mu\in\clp(\R^{d})$, $P\star\mu$ is a function on $\R^{d}$
defined as $P\star\mu(x)=\int_{\R^{d}}P(y,x)\mu(dy)$, $x\in\R^{d}$.
$P'\star\mu$ for $\mu\in\clp(E)$ is defined similarly.
Using Assumption \ref{assu:forthm1} we have \[
\E\left(\Vert P\star(\eta_{k}^{N}-\eta_{k})\Vert_{\infty}\right)\leq\bar{M}_{P}\E\left(\Vert\eta_{k}^{N}-\eta_{k}\Vert_{\infty}\right)\convN0\,.\]
 Finally, for an arbitrary $\delta>0$, we have from Lemma \ref{rem:ascoli}
and Assumption \ref{assu:forthm1} that \begin{eqnarray*}
\Vert P'\star(m_{k}^{N}-m_{k})\Vert_{\infty} & = & \sup_{y\in\R^{d}}\left|\int_{x\in E}(m_{k}^{N}(dx)-m_{k}(dx))P'(x,y)\right|\\
 & \leq & \left(\max_{g\in F_{M_{P'},l_{P'}}^{\delta}}\left|\langle m_{k}^{N}-m_{k},g\rangle\right|+2\delta\right),\,\end{eqnarray*}
 where $F^{\delta}$ is the finite family as in Lemma \ref{rem:ascoli}
associated with $K=E$. Recalling that \eqref{toshow} holds for $k$
and noting that $\delta>0$ is arbitrary and the family $F^{\delta}(M_{P'},l_{P'})$
is finite, we have from the above two estimates that \[
\E(\Vert\eta_{k+1}^{N}-\eta_{k}\Vert_{\infty})\convN0\,.\]
 The result follows. \end{proof} \medskip{}

\begin{proof}[Proof of Corollary \ref{propchaos}] We can apply Proposition
2.2 (i) p. 177 of \cite{sznitman-91} to get that $\forall k,p\in\mathbb{N}$,
$\forall\phi_{1},\dots,\phi_{p}\in C_{b}(E)$, $\E(m_{k}^{N,p}(\phi_{1}\otimes\dots\otimes\phi_{p}))\convN m_{k}^{\otimes p}(\phi_{1}\otimes\dots\otimes\phi_{p})$.
We conclude by a denseness argument. \end{proof}

\subsection{Existence and Uniqueness of Fixed Points.}

In this section we will prove Theorem \ref{cor:unique-fixed-point}.
Some of the techniques in this section are comparable to what can be found in \cite{caron-del-moral-pace-vo-2010}.

For $g:\R^{d}\to\R$, let $\mbox{osc}(g)=\sup_{x,y\in\R^{d}}|g(x)-g(y)|$.
For $n>1$, define $\Phi^{n}$ recursively as $\Phi^{n}=\Phi\circ\Phi^{n-1}$,
where $\Phi^{1}=\Phi$. Let $M_{P,P'}=\max\{M_{P},M_{P'}\}$. We begin
with the following lemma. \begin{lem} \label{lem:Phi-contracting}
Suppose that Assumptions \ref{assu:forthm1}, \ref{assump:Q} and
\ref{assump:P} hold. There are $\eps_{0},\lambda_{0}\in(0,1)$ such
that for any $\eps\le\eps_{0}$ and $\lambda\le\lambda_{0}$ there
exists $\theta\in(0,1)$ such that $\forall n$\[
\Vert(\Phi^{n}(m_{0},\eta_{0})-\Phi^{n}(m_{0}',\eta_{0}')\Vert\leq4\theta^{n-1},\,\]
 for all $(m_{0},\eta_{0}),(m_{0}',\eta_{0}')\in\clp(E)\times\clp^{*}(\R^{d})$.
\end{lem} \begin{proof} Note that from Assumption \ref{assump:P}
(1), for any $(m,\eta)\in\clp(E)\times\clp^{*}(\R^{d})$ \be \Vert\eta
R_{m}-\eta'R_{m'}\Vert{TV}\leq(1-\epsilon)\Vert\eta-\eta'\Vert{TV}+\epsilon\beta(P')\Vert
m-m'\Vert{TV}\,.\label{eq:maj-rec-01}\ee and by Lemma \ref{lem:metropolis-contracting}
in the Appendix, \begin{equation}
\Vert mM^{\eta}-m'M^{\eta}\Vert_{TV}\leq(1-\epsilon_{Q}e^{-\lambda\mbox{osc}(\eta)})\Vert m-m'\Vert_{TV}\,.\label{eq:rec-maj-02}\end{equation}
 Fix $(m_{0},\eta_{0}),(m_{0}',\eta_{0}')\in\clp(E)\times\clp^{*}(\R^{d})$
and let $(m_{n},\eta_{n})=\Phi^{n}(m_{0},\eta_{0})$, $(m'_{n},\eta'_{n})=\Phi^{n}(m'_{0},\eta'_{0})$.
From Assumptions \ref{assu:forthm1} and \ref{assump:P}, we have
$\forall n\geq1$, \begin{equation}
\mbox{osc }(\eta_{n})\leq M_{P,P'}\,,\,\mbox{osc}(\eta'_{n})\leq M_{P,P'}\,.\label{eq:rec-maj-03}\end{equation}
 and therefore for $k\ge2$, \begin{equation}
\Vert m_{k-1}M^{\eta_{k-1}'}-m_{k-1}'M^{\eta_{k-1}'}\Vert_{TV}\leq(1-\epsilon_{Q}e^{-\lambda M_{P,P'}})\Vert m_{k-1}-m'_{k-1}\Vert_{TV}.\label{ins813}\end{equation}

For $k\geq2$ the measures $m_{k-2}P'$ and $m'_{k-2}P'$ have densities
and these densities satisfy $\forall y$\begin{eqnarray}
\left|m_{k-2}P'(y)-m'_{k-2}P'(y)\right| & = & \left|\int_{x\in E}m_{k-2}(dx)P'(x,y)-m'_{k-2}(dx)P'(x,y)\right|\nonumber \\
 & \leq & M_{P,P'}\Vert m_{k-2}-m'_{k-2}\Vert_{TV}\,.\label{eq:maj-densite-01}\end{eqnarray}
 The measures $\eta_{k-2}P$ and $\eta_{k-2}'P$ have densities as
well and these densities satisfy $\forall y$,\begin{eqnarray}
\left|\eta_{k-2}P(y)-\eta_{k-2}'P(y)\right| & = & \left|\int_{x\in E}\left(\eta_{k-2}(dx)P(x,y)-\eta'_{k-2}(dx)P(x,y)\right)\right|\nonumber \\
 & \leq & M_{P,P'}\Vert\eta_{k-2}-\eta_{k-2}'\Vert_{TV}\,.\label{eq:maj-densite-02}\end{eqnarray}
 So, $\forall x,y\in\R^{d}$, using (\ref{eq:rec-eta}), (\ref{eq:maj-densite-01}),
(\ref{eq:maj-densite-02}), \begin{multline}
\left|e^{-\lambda(\eta_{k-1}(x)-\eta_{k-1}(y))_{+}}-e^{-\lambda(\eta_{k-1}'(x)-\eta_{k-1}'(y))_{+}}\right|\leq\lambda|\eta_{k-1}(x)-\eta_{k-1}'(x)|+\lambda|\eta_{k-1}(y)-\eta_{k-1}'(y)|\\
\leq2\lambda M_{P,P'}((1-\epsilon)\Vert\eta_{k-2}-\eta'_{k-2}\Vert_{TV}+\epsilon\Vert m_{k-2}-m'_{k-2}\Vert_{TV})\,.\label{eq:maj-diff-expo}\end{multline}
 By \eqref{eq:rec-maj-04} we then have, for $f\in\clb_{1}(E)$, \begin{multline}
|m_{k-1}M^{\eta_{k-1}}(f)-m_{k-1}M^{\eta_{k-1}'}(f)|\leq4\lambda M_{P,P'}((1-\epsilon)\Vert\eta_{k-2}-\eta'_{k-2}\Vert_{TV}+\epsilon\Vert m_{k-2}-m'_{k-2}\Vert_{TV})\,.\label{eq:rec-maj-04-b}\end{multline}
 By \eqref{ins813}, (\ref{eq:rec-maj-04-b}), we get\begin{multline*}
\Vert m_{k}-m_{k}'\Vert_{TV}\leq\Vert m_{k-1}M^{\eta_{k-1}}-m_{k-1}M^{\eta'_{k-1}}\Vert_{TV}+\Vert m_{k-1}M^{\eta_{k-1}'}-m_{k-1}'M^{\eta_{k-1}'}\Vert_{TV}\\
\leq4\lambda M_{P,P'}((1-\epsilon)\Vert\eta_{k-2}-\eta'_{k-2}\Vert_{TV}+\epsilon\Vert m_{k-2}-m'_{k-2}\Vert_{TV})\\
+(1-\epsilon_{Q}e^{-\lambda M_{P,P'}})\Vert m_{k-1}-m'_{k-1}\Vert_{TV}\,,\end{multline*}
 and combining this with (\ref{eq:maj-rec-01}), we have \begin{eqnarray}
\Vert(m_{k},\eta_{k})-(m'_{k},\eta'_{k})\Vert & \leq & 4\lambda M_{P,P'}((1-\epsilon)\Vert\eta_{k-2}-\eta'_{k-2}\Vert_{TV}+\epsilon\Vert m_{k-2}-m'_{k-2}\Vert_{TV})\nonumber \\
 & + & (1-\epsilon_{Q}e^{-\lambda M_{P,P'}})\Vert m_{k-1}-m'_{k-1}\Vert_{TV}\nonumber \\
 & + & (1-\epsilon)\Vert\eta_{k-1}-\eta_{k-1}'\Vert_{TV}+\epsilon\beta(P')\Vert m_{k-1}-m'_{k-1}\Vert_{TV}\,.\nonumber \\
\label{eq:rec-contraction}\end{eqnarray}
 We can find $\epsilon_{0} \in (0,1)$, $\lambda_{0} \in (0, \infty)$ such that for all $\epsilon \in (0, \epsilon_0)$,
$\lambda \in (0,\lambda_0)$, there exists a   $\theta \equiv \theta(\epsilon, \lambda)\in(0,1)$,
such that \begin{equation}
\frac{\sup((1-\epsilon),(1-\epsilon_{Q}e^{-\lambda M_{P,P'}})+\epsilon\beta(P'))}{\theta}+\frac{4\lambda M_{P,P'}}{\theta^{2}}\leq1\,.\label{ins819}\end{equation}
 Note that \begin{eqnarray*}
\Vert(m_{0},\eta_{0})-(m'_{0},\eta'_{0})\Vert & \leq & 4\theta^{-1}\,,\\
\Vert(m_{1},\eta_{1})-(m'_{1},\eta'_{1})\Vert & \leq & 4\,.\end{eqnarray*}
 We then have by recurrence that for $\epsilon\in(0,\epsilon_{0})$,
$\lambda\in(0,\lambda_{0})$, $\Vert(m_{n},\eta_{n})-(m'_{n},\eta'_{n})\Vert\leq4\theta^{n-1}$,
$\forall n\in\N_{0}$. \end{proof}
With $\epsilon_0, \lambda_0$ and $\theta(\epsilon, \lambda)\equiv \theta$ as in the above
lemma, let $\kappa=\frac{4\lambda M_{P,P'}}{\theta}$. Then from the
estimate in \eqref{ins819} it follows that, for all $k\ge2$ and
$\epsilon\in(0,\epsilon_{0})$, $\lambda\in(0,\lambda_{0})$, \be
\alpha_{k}+\kappa\alpha_{k-1}\ensuremath{\le}\theta\left(\alpha_{k-1}+\kappa\alpha_{k-2}\right),\label{ins1628}
\ee where $\alpha_{k}=\Vert(m_{k},\eta_{k})-(m'_{k},\eta'_{k})\Vert$.
As an immediate consequence we have the following corollary which
will be used in Subsection \ref{uni830}. \begin{cor} \label{firsttwo}
Suppose that Assumptions \ref{assu:forthm1}, \ref{assump:Q} and
\ref{assump:P} hold. Let $\eps_{0},\lambda_{0}\in(0,1)$ be as in
Lemma \ref{lem:Phi-contracting}. Then with $\theta\in(0,1)$ as in
Lemma \ref{lem:Phi-contracting} associated with a fixed choice of
$\eps\le\eps_{0}$, $\lambda\le\lambda_{0}$, we have, for each $k\ge2$
\begin{eqnarray}
 &  & \Vert(m_{k},\eta_{k})-(m'_{k},\eta'_{k})\Vert+\kappa\Vert(m_{k-1},\eta_{k-1})-(m'_{k-1},\eta'_{k-1})\Vert\nonumber \\
 & \le & \theta^{k-1}\left(\Vert(m_{1},\eta_{1})-(m'_{1},\eta'_{1})\Vert+\kappa\Vert(m_{0},\eta_{0})-(m'_{0},\eta'_{0})\Vert\right).\label{eq:rec-contraction-02}\end{eqnarray}
 Suppose further that Assumption \ref{assump:lipQ} holds and that
$m_{0}$ has a density with respect to $\ell_{E}$ that is bounded
by $M_{m_{0}}$. Then, for $k\geq1$, \[
\Vert(m_{k},\eta_{k})-(m'_{k},\eta'_{k})\Vert\le\theta^{k-1}\left(2+\kappa+2\lambda(M_{m_{0}}+M_{Q,Q_{0}})\right)\left(\Vert(m_{0},\eta_{0})-(m'_{0},\eta'_{0})\Vert\right).\]
 \end{cor} \begin{proof} Equation \eqref{eq:rec-contraction-02}
comes directly from \eqref{ins1628}. Next note that \[
\Vert\eta_{1}-\eta'_{1}\Vert_{TV}\leq\Vert m_{0}-m_{0}'\Vert_{TV}+\Vert\eta_{0}-\eta_{0}'\Vert_{TV}\,.\]
 Also, recalling \eqref{eq:rec-maj-04},\begin{multline*}
\Vert m_{0}M^{\eta_{0}}-m_{0}M^{\eta_{0}'}\Vert_{TV}\\
\leq2\int_{x,y\in E}m_{0}(dx)Q(x,dy)\lambda(|\eta_{0}(x)-\eta_{0}'(x)|+|\eta_{0}(y)-\eta_{0}'(y)|)\\
\leq2\lambda(M_{m_{0}}+M_{Q,Q_{0}})\Vert\eta_{0}-\eta_{0}'\Vert_{TV}\,,\end{multline*}
 and\[
\Vert m_{0}M^{\eta_{0}'}-m_{0}'M^{\eta_{0}'}\Vert_{TV}\leq\Vert m_{0}-m_{0}'\Vert_{TV}\,.\]
 Combining these estimates \[
\Vert(m_{1},\eta_{1})-(m_{1}',\eta_{1}')\Vert\leq(1+\sup(1,2\lambda(M_{m_{0}}+M_{Q,Q_{0}}))\Vert(m_{0},\eta_{0})-(m_{0}',\eta_{0}')\Vert\,.\]

\end{proof} \medskip{}

\begin{proof}[Proof of Theorem \ref{cor:unique-fixed-point}]Take
$\epsilon_{0},\lambda_{0}$ as in Lemma \ref{lem:Phi-contracting}
and fix $\epsilon\in(0,\epsilon_{0})$ and $\lambda\in(0,\lambda_{0})$.
The uniqueness is immediate from Lemma \ref{lem:Phi-contracting}.
For existence, take any $(m_{0},\eta_{0})\in\mathcal{P}(E)\times\mathcal{P}^{*}(\R^{d})$
and define recursively $\forall k\geq1$, $(m_{k},\eta_{k})=\Phi(m_{k-1},\eta_{k-1})$.
We have for all $k\geq1$, $p\geq1$, using the $\theta$ given by
Lemma \ref{lem:Phi-contracting}, \begin{eqnarray}
\Vert\Phi^{k+p}(m_{0},\eta_{0})-\Phi^{k}(m_{0},\eta_{0})\Vert & \leq & 4\theta^{k-1}\,.\label{eq:cauchy-sequence}\end{eqnarray}
 So the sequence $(m_{k},\eta_{k})$ is a Cauchy sequence in $(\mathcal{P}(E)\times\mathcal{P}^{*}(\R^{d}),\Vert\cdot\Vert)$
and thus it has a limit $(m_{\infty},\eta_{\infty})$. By taking $p\rightarrow+\infty$
in (\ref{eq:cauchy-sequence}), we see that $(m_{k},\eta_{k})\underset{k\rightarrow+\infty}{\longrightarrow}(m_{\infty},\eta_{\infty})$.
\end{proof}

\subsection{Uniform Convergence over the infinite time horizon.}

\label{uni830} In this subsection we will prove Theorem \ref{pro:unif-conv-theoretical}
and Corollary \ref{cor:appfixpt}.

\begin{proof}[Proof of Theorem \ref{pro:unif-conv-theoretical}]  Fix
$\delta>0$. For $k\in\N_{}$, we take $\Phi_{k+1}^{N}$ to be a (random) operator such that $\Phi_{k+1}^{N}(m_{k}^{N},\eta_{k}^{N})=(m_{k+1}^{N},\eta_{k+1}^{N})$.
By convention, we take $\Phi_{1}^{N}(m_{0},\eta_{0})=(m_{1}^{N},\eta_{1}^{N})$.
Let us denote \[
\mbox{for }1\leq j\,,\,\Phi_{1:j}^{N}(m_{0},\eta_{0})=\Phi_{j}^{N}\circ\Phi_{j-1}^{N}\circ\dots\circ\Phi_{1}^{N}(m_{0},\eta_{0})\,\mbox{ and}\mbox{ for }j=0\,,\,\Phi_{1:j}^{N}(m_{0},\eta_{0})=(m_{0},\eta_{0})\,.\]
 Also, recall that, for $i<j$, $\Phi^{(j-i)}=\Phi\circ\dots\circ\Phi\mbox{ (}j-i\mbox{ times)}$.
We set $\Phi^{(0)}=\mbox{Id}\,.$ Note that for $n\ge1$,\begin{multline}
(m_{n}^{N},\eta_{n}^{N})-(m_{n},\eta_{n})=\sum_{k=1}^{n}\left[\Phi^{(n-k)}\circ\Phi_{k}^{N}\circ\Phi_{1:k-1}^{N}(m_{0},\eta_{0})-\Phi^{(n-k)}\circ\Phi\circ\Phi_{1:k-1}^{N}(m_{0},\eta_{0})\right]\,.\label{eq:telescopic-sum-01}\end{multline}
 We set $\forall k\in\N$, \begin{eqnarray*}
\left(m_{k,k}^{N},\eta_{k,k}^{N}\right)=\Phi_{1:k}^{N}(m_{0},\eta_{0}), &  & (m_{k,k},\eta_{k,k})=\Phi\circ\Phi_{1:k-1}^{N}(m_{0},\eta_{0}),\\
(m_{k,k+1}^{N},\eta_{k,k+1}^{N})=\Phi(m_{k,k}^{N},\eta_{k,k}^{N}), &  & (m_{k,k+1},\eta_{k,k+1})=\Phi(m_{k,k},\eta_{k,k}).\end{eqnarray*}
 Then \begin{multline}
(m_{n}^{N},\eta_{n}^{N})-(m_{n},\eta_{n})=\sum_{k=2}^{n-1}\left(\Phi^{(n-k-1)}(m_{k,k+1}^{N},\eta_{k,k+1}^{N})-\Phi^{(n-k-1)}(m_{k,k+1},\eta_{k,k+1})\right)\\
+\left((m_{n,n}^{N},\eta_{n,n}^{N})-(m_{n,n},\eta_{n,n})\right)\\
+\left(\Phi^{(n-1)}(m_{1}^{N},\eta_{1}^{N})-\Phi^{(n-1)}(m_{1},\eta_{1})\right).\label{ins1705}\end{multline}

Notice that \begin{equation}
\eta_{k,k}^{N}=\eta_{k,k}\,\mbox{ for al }k>1.\label{eq:id-remarquable}\end{equation}
 Now fix a $k\in\{2,\dots,n-1\}$. The signed measure $\eta_{k,k+1}^{N}-\eta_{k,k+1}$
has the following density \[
(\eta_{k,k+1}^{N}-\eta_{k,k+1})(y)=\epsilon\int_{x\in E}(m_{k}^{N}(dx)-m_{k,k}(dx))P'(x,y),\, y\in\R^{d}\,.\]
 From Assumption \ref{leftover} (2), for all $f\in\mathcal{B}_{1}(\R^{d})$,
\[
x\in E\mapsto\int_{y\in\R^{d}}P'(x,dy)f(y)\]
 is $\ti l_{P'}$-Lipschitz. By Lemma \ref{rem:ascoli} we then have,
for every $f\in\clb_{1}(\R^{d})$, $\delta>0$,
\begin{eqnarray}
|\eta_{k,k+1}^{N}(f)-\eta_{k,k+1}(f)| & = & \left|\epsilon\int_{x\in E}(m_{k,k}^{N}(dx)-m_{k,k}(dx))\int_{y\in\R^{d}}P'(x,dy)f(y)\right|\nonumber \\
 & \leq & \epsilon\left(2\delta+\sup_{g\in F_{1}^{\delta}}|\langle m_{k,k}^{N}-m_{k,k},g\rangle|\right)\,.\label{ins343}\end{eqnarray}
 where $F_{1}^{\delta}=F_{1,\ti l_{P'}}^{\delta}(E)$. Now by Lemma
\ref{lem:approx-empirique}, with $C_{1}(\delta)=|F_{1}^{\delta}|$,
\begin{equation}
\E(\Vert\eta_{k,k+1}^{N}-\eta_{k,k+1}\Vert_{TV})\leq2\eps\left(\delta+\frac{C_{1}(\delta)}{\sqrt{N}}\right)\,.\label{eq:maj-phi-01}\end{equation}
 Using (\ref{eq:id-remarquable}) once again \[
m_{k,k+1}^{N}-m_{k,k+1}=m_{k,k}^{N}M^{\eta_{k,k}}-m_{k,k}M^{\eta_{k,k}}\,.\]
 For any $f\in\clb_{1}(E)$, we have \begin{equation}
\langle m_{k,k+1}^{N}-m_{k,k+1},f\rangle=\langle m_{k,k}^{N}-m_{k,k},M^{\eta_{k,k}}f\rangle\,.\label{eq:maj-phi-02-bis}\end{equation}
 From Lemmas \ref{lem:lip-convolution-01} and \ref{lem:lip-convolution-02} in the Appendix
we see that the function $x\in E\mapsto M^{\eta_{k,k}}f(x)$ is $\overline{l}_{Q,Q_{0}}$-Lipschitz
and bounded by $1$ where $\bar{l}_{Q,Q_{0}}=l_{Q,Q_{0}}(3+2\lambda\bar{l}_{P,P'})$.
So using Lemma \ref{rem:ascoli} once again, we have, for $\delta > 0$, \begin{eqnarray}
\E(\Vert m_{k,k+1}^{N}-m_{k,k+1}\Vert_{TV}) & \leq & 2\delta+\E(\sup_{g\in F_{2}^{\delta}}\left|\langle m_{k,k}^{N}-m_{k,k},g\rangle\right|\nonumber \\
 & \leq & 2\left(\delta+\frac{C_{2}(\delta)}{\sqrt{N}}\right)\,.\label{eq:maj-phi-02}\end{eqnarray}
 where $F_{2}^{\delta}=F_{1,\bar{l}_{Q,Q_{0}}}^{\delta}(E)$ and $C_{2}(\delta)=|F_{2}^{\delta}|$.
We will now apply Corollary \ref{firsttwo}. Note that, from Assumption
\ref{assump:lipQ}, for every $k$, $m_{k,k+1}^{N}$ has a density
on $E$ with respect to $l_{E}$ that is bounded by $2M_{Q,Q_{0}}$.
This, in view of Corollary \ref{firsttwo}, along with (\ref{eq:maj-phi-01})
and (\ref{eq:maj-phi-02}) yields $\forall k\in\{2,\dots,n-2\}$,
\begin{multline*}
\E\left(\left\Vert \Phi^{n-k-1}(m_{k,k+1}^{N},\eta_{k,k+1}^{N})-\Phi^{n-k-1}(m_{k,k+1},\eta_{k,k+1})\right\Vert _{TV}\right)\leq\bar{C}\theta^{n-k-2}\left(4\delta+\frac{C_{1}(\delta)+C_{2}(\delta)}{\sqrt{N}}\right)\,,\end{multline*}
 where $\bar{C}=(2+\kappa+6\lambda M_{Q,Q_{0}})$. Note that the above
inequality holds trivially if $k=n-1$. For the term in the second
line of \eqref{ins1705}, note that, for $n>1$, \[
\sup_{f\in\mathcal{B}_{1}(E)}\E(|\langle m_{n,n}^{N}-m_{n,n},f\rangle|)\leq\frac{2}{\sqrt{N}}\,,\,\eta_{n,n}^{N}=\eta_{n,n}\,.\]
 The norm of the term in the third line of \eqref{ins1705}, using
Lemma \ref{lem:Phi-contracting}, can be bounded by $4\theta^{n-2}$.
Combining these estimates, for all $n>1$ \begin{multline*}
\sup_{f\in\mathcal{B}_{1}(E)}\E(|\langle m_{n}^{N}-m_{n},f\rangle|+\Vert\eta_{n}^{N}-\eta_{n}\Vert_{TV})\\
\leq\frac{2}{\sqrt{N}}+\sum_{k=2}^{n-1}\theta^{n-k-2}\bar{C}\left(4\delta+\frac{C_{1}(\delta)+C_{2}(\delta)}{\sqrt{N}}\right)+4\theta^{n-2}\\
\leq\frac{2}{\sqrt{N}}+\bar{C}\left(4\delta+\frac{C_{1}(\delta)+C_{2}(\delta)}{\sqrt{N}}\right)\frac{\theta^{-1}}{1-\theta}+4\theta^{n-1}\,.\end{multline*}
 The result now follows on combining the above estimate with Theorem
\ref{lem:convergence-non-uniforme}. \end{proof}

\begin{proof}[Proof of Corollary \ref{cor:appfixpt}] Fix $\delta>0$.
From Theorem \ref{cor:unique-fixed-point}, there exist $(m_{\infty},\eta_{\infty})\in\mathcal{P}(E)\times\mathcal{P}^{*}(\R^{d})$
and $n_{0}$ such that $\forall n\geq n_{0}$, \[
\Vert(m_{n},\eta_{n})-(m_{\infty},\eta_{\infty})\Vert<\delta\,.\]
 From Theorem \ref{pro:unif-conv-theoretical}, there exist $N_{0},n_{1}\in\N$
such that $\forall n\ge n_{1}$, $\forall N\geq N_{0}$,\global\long\global\long\def\fin{f\in\mathcal{B}_{1}(E)}
 \[
\sup_{\fin}\E(|\langle m_{n}^{N}-m_{n},f\rangle|+\Vert\eta_{n}-\eta_{n}^{N}\Vert_{TV})<\delta\,.\]
 And so\begin{align*}
\limsup_{n\rightarrow+\infty}\limsup_{N\rightarrow+\infty}\sup_{\fin}\E(|\langle m_{n}^{N}-m_{\infty},f\rangle|+\Vert\eta_{n}^{N}-\eta_{\infty}\Vert_{TV}) & <2\delta\,,\\
\limsup_{N\rightarrow+\infty}\limsup_{n\rightarrow+\infty}\sup_{\fin}\E(|\langle m_{n}^{N}-m_{\infty},f\rangle|+\Vert\eta_{n}^{N}-\eta_{\infty}\Vert_{TV}) & <2\delta\,.\end{align*}
 \end{proof}

\subsection{Proof of Theorem \ref{particlefin}. }

In this subsection we will take Assumption \ref{assump:gaussian}
to hold. Recall that under Assumption \ref{assump:gaussian}, we have
that Assumptions \ref{assu:forthm1}, \ref{assump:P} and \ref{leftover}
hold automatically. 

\begin{proof}[Proof of Theorem \ref{particlefin}] We proceed recursively.
Since $\ti\eta_{0}^{N}=\eta_{0}$, we have using Lemma \ref{lem:approx-empirique}
that (\ref{eq:conv-particle-scheme}) holds for $k=0$. Suppose now
that (\ref{eq:conv-particle-scheme}) holds for some $k \in \N_0$. Fix $\delta>0$.
We have $\forall y\in E$,\begin{equation}
\ti\eta_{k+1}^{N}(y)=(1-\epsilon)(P\star S^{N}(\ti\eta_{k}^{N}))(y)+\epsilon(P'\star\ti m_{k}^{N})(y)\,.\label{eq:maj-conv-scheme-00}\end{equation}
 By Assumption \ref{assu:forthm1} and Lemma \ref{rem:ascoli}, we
can write, for $\delta > 0$, \begin{eqnarray}
\E(\sup_{y\in\R^{d}}|(P'\star\ti m_{k}^{N})(y)-(P'\star m_{k})(y)|) & = & \E(\sup_{y\in\R^{d}}|\langle\ti m_{k}^{N}-m_{k},P'(.,y)\rangle|)\nonumber \\
 & \leq & \E(\sup_{g\in F_{3}^{\delta}}|\langle\ti m_{k}^{N}-m_{k},g\rangle|+2\delta)\nonumber \\
 & \leq & \sum_{g\in F_{3}^{\delta}}\E(|\langle\ti m_{k}^{N}-m_{k},g\rangle|)+2\delta,\,\label{eq:maj-conv-scheme-01}\end{eqnarray}
 where $F_{3}^{\delta}=F_{M_{P'},l_{P'}}^{\delta}(E)$. By Lemma \ref{lem:tightness-02},
there exists $K(\delta)$ compact such that
$$\eta_{k}(K(\delta)^{c})<\delta \ , \ \E(\ti\eta_{k}^{N}(K(\delta)^{c}))<\delta \ , \ \forall k,N  \ .
$$ Using
Assumption \ref{assump:P}, \ref{leftover}(1) and Lemma \ref{rem:ascoli},
we can write, \begin{multline}
\E(\sup_{y\in\R^{d}}|(S^{N}(\ti\eta_{k}^{N})P(y)-(\ti\eta_{k}^{N}P)(y)|)\\
\le\E(\sup_{y\in\R^{d}}|\langle S^{N}(\ti\eta_{k}^{N})-\ti\eta_{k}^{N},P(.,y)\1_{K(\delta)}(.)\rangle|+|\langle S^{N}(\ti\eta_{k}^{N})-\ti\eta_{k}^{N},P(.,y)\1_{K(\delta)^{c}}(.)\rangle|)\\
\leq\sum_{g\in F_{4}^{\delta}}\E(|\langle S^{N}(\ti\eta_{k}^{N})-\ti\eta_{k}^{N},g\1_{K(\delta)}(.)\rangle|)+2\delta(1+M_{P})\\
\leq2\left(\frac{C_{4}(\delta)M_{P}}{\sqrt{N}}+\delta(1+M_{P})\right)\,,\label{eq:maj-conv-scheme-02}\end{multline}
 where $F_{4}^{\delta}=F_{M_{P},l_{P}}^{\delta}(K(\delta))$, $C_{4}(\delta)=|F_{4}^{\delta}|$
and the last inequality is a consequence of Lemma \ref{lem:approx-empirique}.

In a similar manner \begin{equation}
\E(\sup_{y\in\R^{d}}|\ti\eta_{k}^{N}P(y)-\eta_{k}P(y)|)\leq\sum_{g\in F_{4}^{\delta}}\E(|\langle\ti\eta_{k}^{N}-\eta_{k},g\1_{K(\delta)}(.)\rangle|)+2\delta(1+M_{P})\,.\label{eq:maj-conv-scheme-02-b}\end{equation}
 And so, by (\ref{eq:maj-conv-scheme-00}), (\ref{eq:maj-conv-scheme-01}),
(\ref{eq:maj-conv-scheme-02}), (\ref{eq:maj-conv-scheme-02-b}),\begin{multline}
\E(\Vert\ti\eta_{k+1}^{N}-\eta_{k+1}\Vert_{\infty})\leq\left(\frac{2C_{4}(\delta)M_{P}}{\sqrt{N}}+6\delta(1+M_{P})\right)+\sum_{g\in F_{3}^{\delta}}\E(|\langle\ti m_{k}^{N}-m_{k},g\rangle|)\\
+2\sum_{g\in F_{4}^{\delta}}\E(|\langle\ti\eta_{k}^{N}-\eta_{k},g\1_{K(\delta)}\rangle|)\,.\label{eq:maj-conv-scheme-03}\end{multline}
 Recalling that (\ref{eq:conv-particle-scheme}) is assumed for $k$,
we have that, as $N\to\infty$, \[
\E(\Vert\ti\eta_{k+1}^{N}-\eta_{k+1}\Vert_{\infty})\to0.\]
 An application of Scheffe's Theorem now shows that, as $N\to\infty$,
\[
\E(\Vert\ti\eta_{k+1}^{N}-\eta_{k+1}\Vert_{TV})\to0.\]
 Next, for any $f\in\clb_{1}(E)$, we have\begin{multline}
\E(|\langle\ti m_{k+1}^{N}-m_{k+1},f\rangle|)\leq\E(|\ti m_{k+1}^{N}-\ti m_{k}^{N}M^{\ti\eta_{k}^{N}},f\rangle|)\\
+\E(|\langle\ti m_{k}^{N}M^{\ti\eta_{k}^{N}}-\ti m_{k}^{N}M^{\eta_{k}},f\rangle|)+\E(|\langle\ti m_{k}^{N}M^{\eta_{k}}-m_{k}M^{\eta_{k}},f\rangle|)\\
\leq\frac{2}{\sqrt{N}}+4\lambda\E(\Vert\ti\eta_{k}^{N}-\eta_{k}\Vert_{\infty})+\E(|\langle\ti m_{k}^{N}-m_{k},M^{\eta_{k}}f\rangle|),\,\label{eq:maj-conv-scheme-05}\end{multline}
 where the last inequality uses (\ref{eq:rec-maj-04}) and Lemma \ref{lem:approx-empirique}.
Once again using the recurrence assumption, we now have that \[
\sup_{f\in\clb_{1}(E)}\E(|\langle\ti m_{k+1}^{N}-m_{k+1},f\rangle|)\convN0.\]
 Thus we have proved that (\ref{eq:conv-particle-scheme}) holds for
$k+1$. The result follows. \end{proof}

\subsection{Proof of Theorem \ref{particleinf}.}

\begin{proof} For $k\in\N$, we take $\overline{\Phi}_{k+1}^{N}$
to be a (random) operator such that $\overline{\Phi}_{k+1}^{N}(\ti m_{k}^{N},\ti\eta_{k}^{N})=(\ti m_{k+1}^{N},\ti\eta_{k+1}^{N})$.
By convention we take $\overline{\Phi}_{1}^{N}(m_{0},\eta_{0})=(\ti m_{1}^{N},\ti\eta_{1}^{N})$.
Following the proof of Theorem \ref{pro:unif-conv-theoretical}, we
define \[
\mbox{for }1\le j\,,\,\overline{\Phi}_{1:j}^{N}(m_{0},\eta_{0})=\overline{\Phi}_{j}^{N}\circ\overline{\Phi}_{j-1}^{N}\circ\dots\circ\overline{\Phi}_{1}^{N}(m_{0},\eta_{0})\,,
\mbox{ and for }j=0\,,\,\overline{\Phi}_{1:j}^{N}(m_{0},\eta_{0})=(m_{0},\eta_{0}).\]
 We define
$\forall k\in\N$,
\beq
\left(m_{k,k}^{N},\eta_{k,k}^{N}\right) &=&\overline{\Phi}_{1:k}^{N}(m_{0},\eta_{0}),\;
(m_{k,k},\eta_{k,k})=\Phi\circ\overline{\Phi}_{1:k-1}^{N}(m_{0},\eta_{0}),\\
(m_{k,k+1}^{N},\eta_{k,k+1}^{N})&=&\Phi(m_{k,k}^{N},\eta_{k,k}^{N}),\;
(m_{k,k+1},\eta_{k,k+1})=\Phi(m_{k,k},\eta_{k,k}).\eeq
 We use the same
symbols as in the proof of Theorem \ref{pro:unif-conv-theoretical}
in order to keep notations simple. We have the following telescopic
decomposition \begin{multline}
(\ti m_{n}^{N},\ti\eta_{n}^{N})-(m_{n},\eta_{n})=\sum_{k=1}^{n}\left[\Phi^{(n-k)}\circ\overline{\Phi}_{k}^{N}\circ\overline{\Phi}_{1:k-1}^{N}(m_{0},\eta_{0})-\Phi^{(n-k)}\circ\Phi\circ\overline{\Phi}_{1:k-1}^{N}(m_{0},\eta_{0})\right]\,.\label{eq:telescopic-sum-01-1}\end{multline}
 The proof is very similar to the proof of Theorem \ref{pro:unif-conv-theoretical},
except that now $\eta_{k,k}^{N}\neq\eta_{k,k}$. The strategy remain
the same; we want to bound the `local error term' $\Phi^{(2)}\circ\overline{\Phi}_{1:k-1}^{N}-\Phi\circ\overline{\Phi}_{1:k}^{N}$
in total variation and then use the contraction property of Lemma
\ref{lem:Phi-contracting} to bound the telescopic sum uniformly in
$N$. Fix $(\eps,\lambda)\in(0,\infty)$ such that $\eps\le\eps_{0}$
and $\lambda\le\lambda_{0}$, where $\eps_{0},\lambda_{0}$  are as
in Lemma \ref{lem:Phi-contracting}. Let $\theta = \theta(\epsilon, \lambda)$. Also, fix $\delta>0$. Consider
a $k\in\{2,\dots,n-1\}$. Note that \[
\eta_{k,k}^{N}-\eta_{k,k}=(1-\epsilon)(S^{N}(\ti\eta_{k-1}^{N})-\ti\eta_{k-1}^{N})P.\]
 Therefore, from \eqref{eq:maj-conv-scheme-02} \begin{equation}
\E(\Vert\eta_{k,k}^{N}-\eta_{k,k}\Vert_{\infty})\leq2\left(\frac{C_{4}(\delta)M_{P}}{\sqrt{N}}+\delta(1+M_{P})\right)\,.\label{eq:maj0unif-part-01}\end{equation}
 In the same way as (\ref{eq:maj-phi-02}) in the proof of Theorem
\ref{pro:unif-conv-theoretical} (see also (\ref{eq:maj-phi-02-bis})),
we get \begin{equation}
\E(\Vert m_{k,k}^{N}M^{\eta_{k,k}}-m_{k,k}M^{\eta_{k,k}}\Vert_{TV})\leq2\left(\delta+\frac{C_{2}(\delta)}{\sqrt{N}}\right)\,.\label{eq:maj-unif-part-02}\end{equation}
 Also from (\ref{eq:rec-maj-04}) and Lemma \ref{lem:expo-lipschitz},
we have $\forall f\in\mathcal{B}_{1}(E)$ \begin{equation}
|\langle m_{k,k}^{N}M^{\eta_{k,k}^{N}}-m_{k,k}^{N}M^{\eta_{k,k}},f>|\leq4\Vert\eta_{k,k}^{N}-\eta_{k,k}\Vert_{\infty}\,.\label{eq:maj-unif-part-03}\end{equation}
 Equations (\ref{eq:maj0unif-part-01}), (\ref{eq:maj-unif-part-02}),
(\ref{eq:maj-unif-part-03}) yield \begin{equation}
\E(\Vert m_{k,k+1}^{N}-m_{k,k+1}\Vert_{TV})\leq\frac{8C_{4}(\delta)M_{P}+2C_{2}(\delta)}{\sqrt{N}}+\delta(10+8M_{P}).\label{eq:maj-unif-part-04}\end{equation}

Next, with $K(\delta)$ as in the proof of Theorem \ref{particlefin},
$\forall f\in\mathcal{B}_{1}(\R^{d})$ \begin{eqnarray}
|\eta_{k,k}^{N}(f)-\eta_{k,k}(f)| & = & \left|(1-\epsilon)\int_{x\in\R^{d},y\in\R^{d}}P(x,y)f(y)dy(S^{N}(\ti\eta_{k-1}^{N})-\ti\eta_{k-1}^{N})(dx)\right|\nonumber \\
 & \leq & \left|\langle S^{N}(\ti\eta_{k-1}^{N})-\ti\eta_{k-1}^{N},Pf(\cdot)\1_{K(\delta)}(\cdot)\rangle\right|+|S^{N}(\ti\eta_{k-1}^{N})(K(\delta)^{c})|\label{eq:maj-unif-part-05}\\
 &  & +|\ti\eta_{k-1}^{N}(K(\delta)^{c})|\,.\nonumber \end{eqnarray}
 Also, using the Gaussian property of the kernel $P$, it follows
that $Pf$ is Lipschitz on $K(\delta)$, uniformly in $f\in\clb_{1}(\R^{d})$.
Denote the uniform bound on the Lipschitz norm by $l_{K(\delta)}$.
Then \begin{equation}
|\langle S^{N}(\ti\eta_{k}^{N})-\ti\eta_{k}^{N},Pf\1_{K(\delta)}\rangle|\le\max_{g\in F_{6}^{\delta}}|\langle S^{N}(\ti\eta_{k}^{N})-\ti\eta_{k}^{N},g\1_{K(\delta)}\rangle|+2\delta,\label{ins130}\end{equation}
 where $F_{6}^{\delta}=F_{1,l_{K(\delta)}}^{\delta}(K(\delta))$.
Thus \[
\E|\eta_{k,k}^{N}(f)-\eta_{k,k}(f)|\leq\E\max_{g\in F_{6}^{\delta}}|\langle S^{N}(\ti\eta_{k-1}^{N})-\ti\eta_{k-1}^{N},g\1_{K(\delta)}\rangle|+4\delta\,.\]
 Next, for all $f\in\clb_{1}(\R^{d})$ \begin{multline*}
\E|\eta_{k,k+1}^{N}(f)-\eta_{k,k+1}(f)|=\E\left|\langle(1-\epsilon)(\eta_{k,k}^{N}-\eta_{k,k})P+\epsilon(m_{k,k}^{N}-m_{k,k})P',f\rangle\right|\\
\leq\E|\langle\eta_{k,k}^{N}-\eta_{k,k},Pf\rangle|+\E|\langle m_{k,k}^{N}P'-m_{k,k}P',f\rangle|\\
\le\E(\max_{g\in F_{6}^{\delta}}|\langle S^{N}(\ti\eta_{k-1}^{N})-\ti\eta_{k-1}^{N},g\1_{K(\delta)}\rangle|)+4\delta+\E\Vert m_{k,k}^{N}P'-m_{k,k}P'\Vert_{TV}\,.\end{multline*}
 Using $C_{1}(\delta)$ introduced in the proof of Theorem \ref{pro:unif-conv-theoretical},
we have \[
\E(\Vert m_{k,k}^{N}P'-m_{k,k}P'\Vert_{TV})\le2\left(\delta+\frac{C_{1}(\delta)}{\sqrt{N}}\right).\]
 We then get  \begin{eqnarray*}
\E(\Vert\eta_{k,k+1}^{N}-\eta_{k,k+1}\Vert_{TV}) & \leq & 2\left(3\delta+\frac{C_{1}(\delta)+C_{6}(\delta)}{\sqrt{N}}\right)\, ,\end{eqnarray*}
where $C_{6}(\delta)=|F_{6}^{\delta}|$.
 Also, by (\ref{eq:maj-unif-part-05}) \[
\E(\Vert\eta_{n,n}^{N}-\eta_{n,n}\Vert_{TV})\leq2\left(2\delta+\frac{C_{6}(\delta)}{\sqrt{N}}\right)\,,\]
 and using the definitions of $m_{n,n}^{N}$ and $m_{n,n}$, we get
from Lemma \ref{lem:approx-empirique} \[
\sup_{f\in\mathcal{B}_{1}(E)}\E(|\langle m_{n,n}^{N}-m_{n,n},f\rangle|)\leq\frac{2}{\sqrt{N}}\,.\]
 Thus as in the proof of Theorem \ref{pro:unif-conv-theoretical}
we get, for $n>1$, \begin{multline*}
\sup_{f\in\mathcal{B}_{1}(E)}\E(|\langle\ti m_{n}^{N}-m_{n},f\rangle|+\Vert\ti\eta_{n}^{N}-\eta_{n}\Vert_{TV})\\
\leq\frac{2}{\sqrt{N}}+2\left(2\delta+\frac{C_{6}(\delta)}{\sqrt{N}}\right)+\bar{C}\left(\delta(3+M_{P})+\frac{2C_{1}(\delta)+2C_{2}(\delta)+2C_{6}(\delta)+8C_{4}(\delta)}{\sqrt{N}}\right)\frac{\theta^{-2}}{1-\theta}+4\theta^{n-1}\,.\end{multline*}
 \end{proof}

\section{Appendix: auxiliary results}

\begin{lem} \label{lem:dobrushin} Let $M$ be a transition probability
kernel on $E$ such that for some $\eps\in(0,1)$ and $\ell_{1}\in\clp(E)$,
\[
M(x,A)\ge\eps\ell_{1}(A),\forall\; A\in\clb(E),x\in E.\]
 Then, for all $\mu_{1},\mu_{2}\in\clp(E)$, \[
\Vert\mu_{1}M-\mu_{2}M\Vert_{TV}\le(1-\eps)\Vert\mu_{1}-\mu_{2}\Vert_{TV}.\]
 \end{lem} \begin{proof} This result comes from Dobrushin's Theorem
and the proof of this theorem can be found in \cite{bartoli-del-moral-2001},
p. 183 and p. 192 (see useful definitions on p. 181). Dobrushin's
Theorem can also be found in \cite{dobrushin-1956-a}, p. 70, with
the corresponding proof in \cite{dobrushin-1956-b}, p. 332.\end{proof}
\begin{lem} \label{lem:metropolis-contracting} Suppose that Assumption
\ref{assump:Q} holds. Then, for any $m,m'\in\clp(\R^{d})$ and $\eta\in\clp^{*}(\R^{d})$,
\[
\Vert mM^{\eta}-m'M^{\eta}\Vert_{TV}\leq(1-\epsilon_{Q}e^{-\lambda\mbox{osc}(\eta)})\Vert m-m'\Vert\,.\]
 \end{lem} \begin{proof} The proof can be found in \cite{bartoli-del-moral-2001},
p. 195. Since this book is written in French, we give a quick proof.
Note that $\forall x,y\in E$,\[
M^{\eta}(x,dy)=Q(x,dy)e^{-\lambda(\eta(x)-\eta(y))_{+}}+Q_{0}(x,dy)\left(1-\int_{z\in E}Q(x,dz)e^{-\lambda(\eta(x)-\eta(z))_{+}}\right)\,.\]
 Then by Assumption \ref{assump:Q}\[
M^{\eta}(x,dy)\geq\epsilon_{Q}e^{-\lambda\mbox{osc}(\eta)}l_{1}(dy)\,.\]
 The result now follows from Lemma \ref{lem:dobrushin}. \end{proof}
\begin{lem} \label{lem:lip-convolution-01} Let $\eta'\in\mathcal{P}(\R^{d})$,
$m\in\mathcal{P}(E)$, and let $\eta=\eta'R_{m}$. Suppose that Assumption
\ref{leftover} (3) holds. Then, $\eta$ has a $\bar{l}_{P,P'}$-
Lipschitz density.\end{lem} \begin{proof} Under Assumption \ref{leftover}
(3), the density of $\eta$ is \[
x\mapsto(1-\epsilon)\int_{y\in\R^{d}}\eta'(dy)P(y,x)+\epsilon\int_{y\in\R^{d}}m(dy)P(y,x).\]
 And for $x,x'\in\R^{d}$,\begin{eqnarray*}
\left|\eta(x)-\eta(x')\right| & \leq & (1-\epsilon)\int_{y\in\R^{d}}\eta'(dy)\bar{l}_{P,P'}|x-x'|+\epsilon\int_{y\in\R^{d}}m(dy)\bar{l}_{P,P'}|x-x'|\\
 & = & \bar{l}_{P,P'}|x-x'|\,.\end{eqnarray*}
 \end{proof} \begin{lem} \label{lem:lip-convolution-02}Take $\eta\in\mathcal{P}(\R^{d})$
having a $l_{\eta}$-Lipschitz density and let $f\in\clb_{1}(\R^{d})$.
Under Assumption \ref{assump:lipQ}, the function $x\in E\mapsto M^{\eta}f(x)$
is $\overline{l}_{Q,Q_{0}}$-Lipschitz with\[
\overline{l}_{Q,Q_{0}}=l_{Q,Q_{0}}\left(3+2\lambda l_{\eta}\right)\,.\]
 \end{lem} \begin{proof} Note that for $x\in E$ \begin{multline*}
M^{\eta}f(x)=\int_{y\in E}Q(x,dy)e^{-\lambda(\eta(x)-\eta(y))_{+}}f(y)\\
+\int_{y\in E}Q_{0}(x,dy)f(y)\times\left(1-\int_{z\in E}Q(x,dz)(1-e^{-\lambda(\eta(x)-\eta(z))_{+}}\right)\,.\end{multline*}
 So for $x,x'\in E$, using Assumption \ref{assump:lipQ} \begin{multline*}
|M^{\eta}f(x)-M^{\eta}f(x')|\leq l_{Q,Q_{0}}|x-x'|\\
+\int_{y\in E}Q(x',dy)|e^{-\lambda(\eta(x)-\eta(y))_{+}}-e^{-\lambda(\eta(x')-\eta(y))_{+}}|\\
+\left|\int_{y\in E}(Q_{0}(x,dy)-Q_{0}(x',dy))f(y)\right|\times\left(1-\int_{z\in E}Q(x,dz)(1-e^{-\lambda(\eta(x)-\eta(z))_{+}}\right)\\
+\int_{y\in E}Q_{0}(x',dy)\times\left|\int_{z\in E}Q(x,dz)e^{-\lambda(\eta(x)-\eta(z))_{+}}-Q(x',dz)e^{-\lambda(\eta(x')-\eta(z))_{+}}\right|\\
\leq l_{Q,Q_{0}}|x-x'|+\int_{y\in E}Q(x',dy)\lambda|\eta(x)-\eta(x')|\\
+l_{Q,Q_{0}}|x-x'|\\
+l_{Q,Q_{0}}|x-x'|+\int_{y\in E}Q(x',dy)\lambda|\eta(x)-\eta(x')|\\
\leq l_{Q,Q_{0}}|x-x'|\left(3+2\lambda l_{\eta}\right)\,.\end{multline*}
 \end{proof} \begin{lem} \label{lem:approx-empirique} (1) Let $E_{0}$
be a closed subset of $\R^{d}$ and let $\mu\in\clp(E_{0})$. Then
for all $f\in\clb_{1}(E_{0})$ \[
\E(|\langle S^{N}(\mu)-\mu,f\rangle|)\leq\frac{2}{\sqrt{N}}.\]
 (2) Let $G$ be a transition probability kernel on a closed subset
$E_{0}$ of $\R^{d}$. Let $x_{1},\cdots x_{N}\in E_{0}$ for some
$N\in\N$ and let $\xi_{1},\cdots\xi_{N}$ be mutually independent
random variables distributed as $\delta_{x_{1}}G,\cdots\delta_{x_{N}}G$.
Define $m_{0}^{N}=\frac{1}{N}\sum_{i=1}^{N}\delta_{x_{i}}$ and $m_{1}^{N}=\frac{1}{N}\sum_{i=1}^{N}\delta_{\xi_{i}}$.
Then for all $f\in\clb_{1}(E_{0})$ \[
\E(|\langle m_{1}^{N}-m_{0}^{N}G,f\rangle|)\le\frac{2}{\sqrt{N}}.\]
 \end{lem} 

\begin{proof} We will only show (2). Proof of (1) is similar. Fix
$f\in\clb_{1}(E_{0})$. Then \begin{eqnarray*}
\E(|\langle m_{1}^{N}-m_{0}^{N}G,f\rangle|) & = & \E\left|\frac{1}{N}\sum_{i=1}^{N}f(\xi_{i})-\delta_{x_{i}}G(f)\right|\\
 & \le & \left(\E\left(\frac{1}{N}\sum_{i=1}^{N}f(\xi_{i})-\delta_{x_{i}}G(f)\right)^{2}\right)^{1/2}\\
 & = & \left(\E\frac{1}{N^{2}}\sum_{i=1}^{N}\left(f(\xi_{i})-\delta_{x_{i}}G(f)\right)^{2}\right)^{1/2}\\
 & \le & \frac{2}{\sqrt{N}}.\end{eqnarray*}
 \end{proof} \begin{lem} \label{lem:consistancy-of-assumptions}Suppose
$P,P'$ satisfy Assumption \ref{assump:gaussian}, then they satisfy
Assumption \ref{assump:P}(1).\end{lem} \begin{proof} Let $m_{0},m_{0}'\in\mathcal{P}(E)$.
We can write \[
\begin{cases}
m_{0}= & m+\tilde{m}\\
m_{0}'= & m+\tilde{m}'\end{cases}\]
 where $m,\ti m,\ti m'$ are finite measures on $E$ and $\Vert m_{0}-m_{0}'\Vert_{TV}=\Vert\tilde{m}-\tilde{m}'\Vert_{TV}=\Vert\tilde{m}\Vert+\Vert\tilde{m}'\Vert_{TV}$.
We set \[
\begin{cases}
P'_{E}(x,dx')= & P'(x,dx')\1_{x'\in E}\\
P'_{E^{c}}(x,dx')= & P'(x,dx')\1_{x'\in E^{c}}\,.\end{cases}\]
 Since $E$ is compact, we can find $p\in\clp(E)$ and $\delta>0$
such that \begin{equation}
P'_{E}(x,dx')=\delta p(dx')+M(x,dx')\,,\,\forall x\in E.\label{ins512}\end{equation}
 where $M$ is some nonnegative kernel. Then \begin{eqnarray*}
\Vert m_{0}P'-m_{0}'P'\Vert_{TV} & = & \Vert\tilde{m}P'_{E}-\tilde{m}'P'_{E}\Vert_{TV}+\Vert\tilde{m}P'_{E^{c}}-\tilde{m}'P'_{E^{c}}\Vert_{TV}\\
 & \leq & \Vert\tilde{m}M-\tilde{m}'M\Vert_{TV}+\tilde{m}P'(E^{c})+\tilde{m}'P'(E^{c})\\
 & \leq & \tilde{m}M(E)+\tilde{m}'M(E)+\tilde{m}P'(E^{c})+\tilde{m}'P'(E^{c})\,.\end{eqnarray*}
 From \eqref{ins512} we have \begin{eqnarray*}
\tilde{m}M(E)+\tilde{m}P'(E^{c}) & = & \tilde{m}P'(E)-\delta\tilde{m}(E)+\tilde{m}P'(E^{c})\\
 & \leq & (1-\delta)\tilde{m}P'(E)+\tilde{m}P'(E^{c})\\
 & = & (1-\delta)\tilde{m}P'(E)+(\ti m(E)-\tilde{m}P'(E)).\end{eqnarray*}
 Since $E$ has a non-empty interior, we can find $\alpha>0$ such
that $\forall\overline{m}\in\mathcal{P}(E)$, $\overline{m}P'(E)\geq\alpha$.
In particular we have $\tilde{m}P'(E)\geq\alpha\tilde{m}(E)$. Using
this property in the above display \begin{eqnarray*}
\tilde{m}M(E)+\tilde{m}P'(E^{c}) & \leq & (1-\delta)\alpha\tilde{m}(E)+(1-\alpha)\tilde{m}(E)\\
 & = & (1-\alpha\delta)\tilde{m}(E)\,,\end{eqnarray*}
 Also, the same formula holds with $\tilde{m}$ replaced by $\tilde{m}'$.
So \begin{eqnarray*}
\Vert m_{0}P'-m_{0}'P'\Vert_{TV} & \leq & (1-\alpha\delta)(\tilde{m}(E)+\tilde{m}'(E))\\
 & = & (1-\alpha\delta)\Vert m_{0}-m_{0}'\Vert_{TV}\,.\end{eqnarray*}
 \end{proof}

\begin{lem} \label{lem:tightness-02} Suppose that Assumption \ref{leftover}
(\ref{enu:lip-P-01}) holds. Fix $(m_{0},\eta_{0})\in\clp(E)\times\clp^{*}(\R^{d})$.
Let $\eta_{k}$ be as defined in \eqref{eq:rec01} and let $\ti\eta_{k}^{N}$
be as defined in \eqref{eq:def-scheme}. Then, for any $\delta>0$,
$\exists$ a compact subset of $\R^{d}$, $K(\delta)$, such that
$\forall k$, $\forall N$,\[
\eta_{k}(K(\delta)^{c})<\delta\,,\,\E(\ti\eta_{k}^{N}(K(\delta)^{c})<\delta\,.\]
 \end{lem} \begin{proof}  We can prove by recurrence
that $\forall n$, \[
\eta_{n+1}=\sum_{k=0}^{n}\left[\epsilon(1-\epsilon)^{k}m_{n-k}P'P^{k}\right]+(1-\epsilon)^{n+1}\eta_{0}P^{n+1}\,.\]
 Fix $\ti\delta>0$. Let $k_{0}$ such that $\sum_{k\geq k_{0}+1}(1-\epsilon)^{k}<\ti\delta$.
Since $E$ is compact, we can find a finite family $(x_{j})_{j\in J}$
of elements of $E$ such that $\forall y\in E$, $\exists j\in J$
such that $|y-x_{j}|<\ti\delta$. We can partition: $E=\sqcup_{j\in J}E_{j}$
such that $\forall j$ and $\forall y\in E_{j}$, $|y-x_{j}|<\ti\delta$.
For each $j\in J$, we can find $K_{j}$, a compact subset of $\R^{d}$,
such that $\delta_{x_{j}}P'(K_{j}^{c})<\ti\delta,\,\delta_{x_{j}}P'P(K_{j}^{c})<\ti\delta,\dots,\delta_{x_{j}}P'P^{k_{0}}(K_{j}^{c})<\ti\delta$.
Clearly, the set $K=\cup_{j\in J}K_{j}$ is compact. Also, $\forall k\leq k_{0}$,
$\forall r$\begin{eqnarray*}
m_{r}P'P^{k}(K^{c}) & = & \int_{y\in E}m_{r}(dy)(P'P^{k})(y,K^{c})\\
 & = & \sum_{j\in J}\int_{y\in E_{j}}m_{r}(dy)(P'P^{k})(y,K^{c})\\
 & = & \sum_{j\in J}\int_{y\in E_{j}}m_{r}(dy)(P'P^{k})(x_{j},K^{c})\\
 &  & +\sum_{j\in J}\int_{y\in E_{j}}m_{r}(dy)((P'P^{k})(y,K^{c})-(P'P^{k})(x_{j},K^{c}))\\
 & \leq & \sum_{j\in J}\left[m_{r}(E_{j})(P'P^{k})(x_{j},K^{c})\right]+\ti\delta\ti l_{P,P'}\\
 & \leq & \ti\delta+\ti\delta\ti l_{P,P'}\,,\end{eqnarray*}
 where the next to last inequality follows from Assumption \ref{leftover}
(2). Finally, $\forall n\geq k_{0}+1$ \begin{eqnarray*}
\eta_{n+1}(K^{c}) & \leq & \sum_{k=0}^{k_{0}}[\epsilon(1-\epsilon)^{k}m_{n-k}P'P^{k}(K^{c})]+\ti\delta\\
 & \leq & \sum_{k=0}^{k_{0}}[\epsilon(1-\epsilon)^{k}(\ti\delta+\ti\delta\ti l_{P,P'})]+\ti\delta\,.\end{eqnarray*}
 The first statement follows. To prove the second statement, we begin
by defining the (random )operator $S^{N}\circ P$ acting on probability measures
on $\R^{d}$: $\mu(S^{N}\circ P)=(S^{N}(\mu))P$. We have\begin{eqnarray*}
\tilde{\eta}_{n+1}^{N} & = & \sum_{k=0}^{n}\left[\epsilon(1-\epsilon)^{k}\tilde{m}_{n-k}^{N}P'(S^{N}\circ P)^{k}\right]+(1-\epsilon)^{n+1}\eta_{0}(S^{N}\circ P)^{n+1}\,.\end{eqnarray*}
We take the same $\tilde{\delta},x_{j},K_{j},\dots$ as above and
we notice that $\forall k,j$, \[
\E(\delta_{x_{j}}P'(S^{N}\circ P)^{k}(K_{j}^{c}))=\delta_{x_{j}}P'P^{k}(K_{j}^{c})\,.\]
From this point, the proof is the same as for the first statement.\end{proof}

\begin{acknowledgements}AB has been supported in part by the National
Science Foundation (DMS-1004418), the Army Research Office (W911NF-0-1-0080,
W911NF-10-1-0158) and the US-Israel Binational Science Foundation
(2008466). Part of this work was done when AB was visiting Universit\'e de Nice-Sophia Antipolis and Universit\'e de Bordeaux I and when SR was visiting University of North Carolina at Chapel Hill. Hospitality of these universities is gratefully acknowledged. \end{acknowledgements}
\bibliographystyle{plain}
\bibliography{bib-invariant-measures}

\end{document}